\documentclass[12pt]{article}
\usepackage{mathpazo}
\usepackage{times}

\usepackage{amsthm}
\usepackage{amsmath}
\usepackage{amsfonts}
\usepackage{graphicx}
\usepackage{latexsym}
\usepackage{amssymb}
\usepackage{color}

\newcommand{\field}[1]{\mathbb{#1}}
\newcommand{\B}{\field{B}}

\newcommand{\Z}{\field{Z}}

\newcommand{\dS}{\field{S}}
\newcommand{\dT}{\field{T}}
\newcommand{\dP}{\field{P}}
\newcommand{\cA}{{\cal A}}
\newcommand{\cB}{{\cal B}}
\newcommand{\cC}{{\cal C}}
\newcommand{\cD}{{\cal D}}
\newcommand{\cF}{{\cal F}}
\newcommand{\cG}{{\cal G}}

\newcommand{\cS}{{\cal S}}
\newcommand{\cL}{{\cal L}}
\newcommand{\cT}{{\cal T}}
\newcommand{\cP}{{\cal P}}

\newcommand{\cR}{{\cal R}}

\newcommand{\cU}{{\cal U}}

\newcommand{\cM}{{\cal M}}

\newcommand{\cH}{{\cal H}}

\newtheorem{theorem}{Theorem}
\newtheorem{lemma}{Lemma}
\newtheorem{example}{Example}
\newtheorem{remark}{Remark}
\newtheorem{corollary}{Corollary}

\date{$\,$\\[2.70ex]\normalsize\large\today}

\makeatletter

\gdef\@punct{.\ \ }  
\def\@sect#1#2#3#4#5#6[#7]#8{%
  \ifnum #2>\c@secnumdepth
     \def\@svsec{}
  \else
     \refstepcounter{#1}\edef\@svsec{%
     \ifnum #2>0{{\csname the#1\endcsname}}.\fi%
    \hskip .5em}
  \fi
  \@tempskipa #5\relax
  \ifdim \@tempskipa>\z@
     \begingroup #6\relax
       \@hangfrom{\hskip #3\relax\@svsec}{\interlinepenalty \@M #8\par}
     \endgroup
     \csname #1mark\endcsname{#7}
     \addcontentsline{toc}{#1}{\ifnum #2>\c@secnumdepth\else
          \protect\numberline{\csname the#1\endcsname}\fi#7}
  \else
     \def\@svsechd{#6\hskip #3\@svsec #8\@punct\csname #1mark\endcsname{#7}
     \addcontentsline{toc}{#1}{\ifnum #2>\c@secnumdepth \else
          \protect\numberline{\csname the#1\endcsname}\fi#7}}
  \fi
  \@xsect{#5}}

\def\@ssect#1#2#3#4#5{\@tempskipa #3\relax
  \ifdim \@tempskipa>\z@
    \begingroup #4\@hangfrom{\hskip #1}{\interlinepenalty \@M #5\par}\endgroup
  \else \def\@svsechd{#4\hskip #1\relax #5\@punct}\fi
  \@xsect{#3}}

\makeatother

\begin{document}

\bibliographystyle{plain}

\title{
\bf Explicit Baranyai Partitions for Quadruples,\\
Part I: Quadrupling Constructions\\[1.80ex]
}
\author{
{\sc Yeow Meng Chee}%
\thanks{Department of Industrial Systems Engineering and Management,
National University of Singapore, Singapore, e-mail: {\tt pvocym@nus.edu.sg}.}
\and
{\sc Tuvi Etzion}%
\thanks{Department of Computer Science, Technion,
Haifa 3200003, Israel, e-mail: {\tt etzion@cs.technion.ac.il}.
Tuvi Etzion was supported in part by the Binational Science Foundation (BSF) grant no. 2018218}
\and
{\sc Han Mao Kiah}%
\thanks{School of Physical and Mathematical Sciences,
Nanyang Technological University, Singapore, e-mail: {\tt hmkiah@ntu. edu.sg}.}
\and
{\sc Alexander Vardy}%
\thanks{Department of Electrical \& Computer Engineering and
Department of Computer Science \& Engineering,
University of California San Diego, La Jolla, CA 92093, USA,
e-mail: {\tt avardy@ucsd.edu}.}
\and
{\sc Chengmin Wang}%
\thanks{School of Science, Taizhou University, Taizhou 225300, China,
e-mail: {\tt chengmin\_wang@163.com}.
Chengmin Wang was supported by Natural Science Foundation
of Jiangsu (Grant No. BK20171318), Qing Lan Project of Jiangsu
and Overseas Study Program of Jiangsu University.}
}

\maketitle
\thispagestyle{empty}

\begin{abstract}
\noindent
It is well known that, whenever $k$ divides $n$,
the complete $k$-uniform hypergraph on $n$ vertices
can be partitioned into disjoint perfect matchings.
Equivalently, the set of $k$-subsets of an $n$-set 
can be partitioned into parallel classes so that each
parallel class is a partition of the $n$-set. 
This result~is~known as Baranyai's theorem, which guarantees
the existence of \emph{Baranyai partitions}.
Unfortunately, the proof of Baranyai's theorem uses network flow
arguments, making this result non-explicit. In particular,~there
is no known method to produce Baranyai partitions in time and space
that scale linearly with the number of hyperedges in the hypergraph.
It is desirable for certain applications to have an explicit
construction that generates Baranyai partitions in linear time.
Such an efficient construction is known for $k=2$ and $k=3$.
In this paper, we present an explicit recursive quadrupling
construction for $k=4$ and $n=4t$, where
$t \equiv 0,3,4,6,8,9 ~(\text{mod}~12)$.
In a follow-up paper (Part II), the other values of~$t$, namely
$t \equiv 1,2,5,7,10,11 ~(\text{mod}~12)$,
will be considered.
\end{abstract}

\vspace{0.5cm}

\noindent
\centerline{{\bf Keywords:} Baranyai partitions, Baranyai's theorem, near-one-factorization, one-factorization}


\newpage
\section{Introduction}
\label{sec:introduction}

A \emph{hypergraph} is a pair $G=(V,E)$, where $V$ is a set of
elements, called \emph{vertices}, and $E$ is a set~of~non\-empty
subsets of $V$, called {\em hyperedges} or \emph{edges}. The size
of the vertex set is called the \emph{order} of~the~hyper\-graph.
A \emph{$k$-uniform hypergraph} is a hypergraph such that all edges
in $E$ have size $k$.  A 2-uniform~hypergraph is also called a
\emph{graph}.
Let $\binom{V}{k}$ denote the set of all $k$-subsets of $V$.
A hypergraph $(V,\binom{V}{k})$~is~called a
\emph{complete $k$-uniform hypergraph}.  A \emph{one-factor} or
a~{\em parallel class} in a $k$-uniform hypergraph is a~set of 
disjoint edges that partition the set of vertices $V$.
A \emph{one-factorization} of a hypergraph $G=(V,E)$~is~a~pair $(V,F)$,
where $F$ is a set of parallel classes that partition the
set of edges $E$.

Given a complete $k$-uniform hypergraph 
of order $n$, all of its
its edges can be partitioned~into parallel classes.
Of course, this requires that $k$ divides $n$, and the 
number of parallel classes is 
$\binom{n}{k}/{\frac{n}{k}}={\binom{n-1}{k-1}}$.
This result is well known as {Baranyai's theorem};
the resulting
partitions are called Baranyai partitions.

More precisely,
given a set $V$ of cardinality $n$,
an \emph{$(n,k)$-Baranyai partition of $V$},
denoted $\mathrm{BP}(n,k)$, is a one-factorization
of the complete $k$-uniform hypergraph \smash{$G=(V,\binom{V}{k})$}.
The existence of $\mathrm{BP}(n,2)$~is known since the 19th century;
extensive study of related one-factorizations can be found in~\cite{Wal97}.
Explicit and efficient recursive constructions of $\mathrm{BP}(n,3)$,
for every $n$ divisible by~$3$, were
given 
by Peltesohn~\cite{Pel36} in her doctoral thesis from 1936.
Bermond proved the existence of $\mathrm{BP}(n,4)$ in the early 1970s,
using coloring techniques from graph theory, but his result was never
published~\cite[p.\,537]{LiWi01}.
In fact, after~almost fifty 
years, it does not appear possible
to reconstruct the proof in its entirety~\cite{Ber19}, and it is
not completely clear which values of $n$ have been solved.
Moreover, it seems that the techniques used by Bermond would
not lead to a construction whose complexity scales linearly
in the number of hyperedges.

The general case was resolved by Baranyai~\cite{Bar75} in 1975,
using network flow methods. 
Specifically,~Ba\-ra\-nyai~\cite{Bar75}
established the following existence result.
\begin{theorem}
\label{Baranyai theorem}
If\/ $k$ and $n$ are positive integers,
then\/ $\mathrm{BP}(n, k)$ exists whenever $k\mid n$.
\end{theorem}
\noindent
The proof of the foregoing theorem of Baranyai involves several
iterations of a network flow algorithm. Unfortunately, even today,
the best-known algorithm~\cite{Orl13} for computing the maximum
flow in a network $G=(V,E)$ requires $O(|V|{\cdot}|E|)$ time.
The complexity of producing Baranyai partitions
is even higher. For example, the iterative implementation
given by Brouwer and Schrijver~\cite{BrSc79}, 
based on network flow techniques, has complexity $O(n^{3k-2})$ for a fixed $k$.
However, for certain applications, it would be desirable to generate 
Baranyai partitions in time and space that scale \emph{linearly with the
number of hyperedges}.~~

Hence, we say that a construction of $\mathrm{BP}(n,k)$ is \emph{efficient}
if its complexity is $O\bigl(\binom{n}{k}\bigr)$.
Our goal herein is to provide 
an explicit and efficient construction of Baranyai partitions for $k = 4$.
Our results also~lead,\linebreak
via enumerative coding, to efficient algorithms
for 
generating a specific parallel class in $\mathrm{BP}(n,4)$
or~a~specific hyperedge in the parallel class.
The complexity of these algorithms is $O(n)$.

To this end, we introduce efficient recursive \emph{quadrupling constructions}
for $\mathrm{BP}(4t,4)$. Specifically,~we will show how to construct
$\mathrm{BP}(4t,4)$, given
$\mathrm{BP}(t+\delta_1,4)$ and $\mathrm{BP}(t+\delta_2,3)$,
where 
$0 \leq \delta_1 \leq 3$ and $0 \leq \delta_2 \leq 2$.
In this paper, we provide such recursive constructions
for $t \equiv 0,3,4,6,8,9 \!\pmod{12}$.
In other words, we prove the following theorem.
\begin{theorem}
\label{thm:main}
If\kern1pt\ $t \equiv 0,3,4,6,8,9 \!\pmod{12}$
and there exists an efficient construction for $\mathrm{BP}(t+\delta,4)$,
where $0 \leq \delta \leq 3$,
then there exists an efficient construction for $\mathrm{BP}(4t,4)$.
\end{theorem}

\looseness=-1
In a follow-up paper~\cite{CEKVW20} we consider the remaining cases,
namely $t \equiv 1,2,5,7,10,11 \!\pmod{12}$.
Together with Theorem\,\ref{thm:main}, this makes it possible to
efficiently construct $\mathrm{BP}(n,4)$ for every $n$ divisible by $4$.

\looseness=-1
The rest of this paper is organized as follows. We begin in
Section\,\ref{sec:designs} with some background~on~combinatorial
designs, which is necessary for our constructions. The discussion
in Section\,\ref{sec:designs} briefly covers
Steiner systems, large sets of Steiner systems, resolvable Steiner systems,
one-factorizations, near-one-factorizations, and Latin squares.
In Section\,\ref{sec:doubling}, a doubling construction for
$\mathrm{BP}(2t,4)$ from $\mathrm{BP}(t,4)$~is~presented
for the case where $t \equiv 4$ or $8 \!\pmod{12}$.
This construction is based on resolvable Steiner~quadruple systems
and one-factorizations of the complete graph $K_t$.
In Section\,\ref{sec:general}, the ingredients
for our quad\-rupling constructions are discussed.
First, the set of quadruples is partitioned into five {groups}.
Then~the parallel classes of the Baranyai partitions we construct
are classified into five {types},
depending upon which groups take part in the parallel class.
In Section\,\ref{sec:const_par}, we present constructions for each
of the five types of parallel classes 
introduced in Section\,\ref{sec:general}.
For three of the five types, the construction depends on whether $t$
is even or odd. For one type, the construction depends on the value
of $t$ modulo $4$. Finally, for one of the types we will restrict ourselves,
in this section, to values of $t$ that are divisible by $3$.
In Section\,\ref{sec:main}, we summarize the results from
Section\,\ref{sec:const_par} to present a quadrupling construction
of $\mathrm{BP}(4t,4)$ for
$t \equiv 0,3,4,6,8,9 \!\pmod{12}$.
The case $t \equiv 9 \!\pmod{12}$
is more involved and requires further~modification
of the constructions given in Section\,\ref{sec:general}.
Section\,\ref{sec:efficient} is devoted to complexity analysis of
several problems associated with Baranyai partitions and constructions
thereof. As an example of our methods,~we present in Section\,\ref{sec:efficient}
an efficient enumerative-coding algorithm for constructing
$\mathrm{BP}(2^n,4)$ for all $n \geq 2$.\linebreak
We conclude with a brief discussion in Section\,\ref{sec:conclude}.

\vspace{1.80ex}\noindent
{\bf Note on terminology.}
For the sake of
brevity, in what follows, whenever
we say \emph{there exists~a~$\mathrm{BP}(n,k)$} we mean that there also exists
an efficient construction for it with complexity $O\bigl(\binom{n}{k}\bigr)$.

\section{Background on Combinatorial Designs}
\label{sec:designs}

Our constructions will make use of some well-known combinatorial designs which
will be described in the following subsections.

\subsection{Steiner Systems, Large Sets, and Resolvable Designs}
\label{sec:Steiner}

A \emph{Steiner system} $\mathrm{S}(t,k,v)$ is a pair $(Q,\B)$, where $Q$ is a $v$-set of \emph{points}
and $\B$ is a set of $k$-subsets (called \emph{blocks}) of $Q$, such that each $t$-subset of $Q$
is contained in exactly one block of $\B$. The literature on Steiner systems is very rich (many book chapters and surveys,
see for example~\cite{CoMa06}). Our interest is only in Steiner systems $\mathrm{S}(k-1,k,v)$,
Steiner systems $\mathrm{S}(1,k,v)$, and Steiner systems $\mathrm{S}(k,k,v)$. Note, that $\mathrm{S}(k,k,v)$ is equivalent to the
complete $k$-uniform hypergraph on $v$ vertices.

A \emph{large set} of Steiner systems $\mathrm{S}(t,k,v)$ is a partition of all the $\binom{v}{k}$ $k$-subsets
of $Q$ into Steiner systems $\mathrm{S}(t,k,v)$. In our context, a $\mathrm{BP}(n,k)$ is a large set of Steiner
systems $\mathrm{S}(1,k,v)$. Large sets were extensively studied during the years. For example, a large set of
$\mathrm{S}(2,3,v)$ is known for each $v \equiv 1$ or $3~(\text{mod}~6)$~\cite{Ji05,Lu83,Lu84,Tei91}. No large set of
$\mathrm{S}(3,4,v)$ is known, but many pairwise disjoint sets of $\mathrm{S}(3,4,v)$ were found, e.g.~\cite{EtHa91,EtZh20}.

A Steiner system $\mathrm{S}(k-1,k,v)$ is \emph{resolvable} if its block set $\B$ can be partitioned into
subsets $\cB_1,\cB_2,\ldots$, such that each subset $\cB_i$ is a Steiner system $\mathrm{S}(1,k,v)$, i.e.
each $\cB_i$ is a subset of disjoint blocks whose union is the point set $Q$. It is important
to note that $\mathrm{BP}(n,k)$ is a resolvable Steiner system $\mathrm{S}(k,k,n)$.
Of special interest in our exposition are two types of resolvable Steiner systems.
A Steiner system $\mathrm{S}(3,4,v)$ is called a \emph{Steiner quadruple system} and it is denoted by $\mathrm{SQS}(v)$. Such a system exists
if and only if $v \equiv 2$ or $4~(\text{mod}~6)$~\cite{Han60}. It is known that there exists an efficient
construction for resolvable $\mathrm{SQS}(v)$
if and only if $v \equiv 4$ or $8~(\text{mod}~12)$~\cite{Har87,LiZh05}.

The number of parallel classes in a $\mathrm{BP}(n,k)$ will be used in our exposition to prove the
correctness of the constructions which follow.
The total number of $k$-subsets of an $n$-set is $\binom{n}{k}$ and a parallel class contains $\frac{n}{k}$
$k$-subsets. This implies the following theorem.

\begin{theorem}
\label{thm:num_classes}
The number of parallel classes in a $\mathrm{BP}(n,k)$ is $\binom{n-1}{k-1}$.
\end{theorem}

Finally, we will use a $\mathrm{BP}(n,2)$ which is a resolvable Steiner system $\mathrm{S}(2,2,n)$, also called
\emph{one-factorization of the complete graph $K_v$}. This structure is discussed in the next subsection.

\subsection{One-Factorization and Near-One-Factorization}
\label{sec:one_near}

A \emph{one-factorization} of the complete graph $K_m$, $m$ even, is a partition of the
edges of $K_m$ into perfect matchings. In other words, the set
$$
\cF = \{ \cF_0, \cF_1 , \ldots , \cF_{m-2} \}
$$
is a one-factorization of $K_m$ if each $\cF_i$, $0 \leq i \leq m-2$, is a perfect matching
(called a \emph{one-factor}), and the $\cF_i$'s are pairwise disjoint.
The $\frac{m}{2}$ pairs of $\cF_i$ are ordered arbitrarily,
$\cF_{i,0},\cF_{i,1},\ldots,\cF_{i,(m-2)/2}$.

If $m$ is odd, then there is no perfect matching in $K_m$ and we define a \emph{near-one-factorization}
$$
\cF = \{ \cF_0, \cF_1 , \ldots , \cF_{m-1} \}
$$
to be a partition of the edges in $K_m$ into sets of $\frac{m-1}{2}$ pairwise disjoint edges,
where each $\cF_i$ has one isolated vertex. Each $\cF_i$ is called a \emph{near-one-factor}.
In the sequel we assume that the set of vertices of $K_m$ is $\Z_m$ and in the near-one-factor $\cF_i$, $0 \leq i \leq  m-1$,
the vertex $i$ is isolated. The $\frac{m-1}{2}$ pairs of $\cF_i$ are ordered arbitrarily,
$\cF_{i,0},\cF_{i,1},\ldots,\cF_{i,(m-3)/2}$.

As was noted before, there has been done extensive study on
one-factorizations and near-one-factor\-iza\-tion
and an excellent book for this is~\cite{Wal97}.

\begin{example}
\label{ex:factors}
For $m=3$ there are three near-one-factors in a near-one-factorization $\cF = \{ \cF_0 ,\cF_1 , \cF_2 \}$,
where $\cF_0 = \{ \{ 1,2 \} \}$, $\cF_1 = \{ \{ 0,2 \} \}$, and $\cF_2 = \{ \{ 0,1 \} \}$.
For $m=4$, this yields a one-factorization $\cF' = \{ \cF'_0 ,\cF'_1 , \cF'_2 \}$,
where $\cF'_0 = \{ \{ 1,2 \} , \{ 0,3 \} \}$, $\cF'_1 = \{ \{ 0,2 \}, \{ 1,3 \} \}$, and $\cF'_2 = \{ \{ 0,1 \} , \{ 2,3 \} \}$.
\end{example}

The examples for $m=3$ and $m=4$ are special cases of a general construction.
Let $m > 1$ be a positive odd integer. We define the near-one-factor $\cF_k$, $0 \leq k \leq m-1$, as follows.
$$
\cF_k \triangleq \{ \{x,y\} ~:~ x,y \in \Z_m,~ x \neq y, ~ x+y \equiv 2k ~(\text{mod}~m)    \}~.
$$
The one-factorization $\cF'$ of $K_{m+1}$ is defined from the near-one-factorization $\cF$ of $K_m$
as follows
$$
\cF'_k \triangleq \cF_k \cup \{ \{ k,m \} \}.
$$

\subsection{Latin Squares}
\label{sec:Latin}

All the designs which were previously discussed consist of blocks and hence they are
also called \emph{block designs}. The last combinatorial design which will be used in
our constructions is not a block design, but it is heavily used in constructions
of various types of block designs including Steiner systems.

A \emph{Latin square} of order $n$ is an $n \times n$ matrix, say $M$, with entries from
a given $n$-set $Q$ such that each row and each column of $M$ is a permutation of the
elements of $Q$. An important result in our discussion, which can be proved for example by using
the well-known Hall's marriage Theorem~\cite{Hal35} is given next.

\begin{theorem}
\label{thm:completLS}
If a $k \times n$ matrix $\hat{M}$ has rows which are permutations of an $n$-set $Q$ and columns
with distinct elements of $Q$, then $\hat{M}$ can be completed to an $n \times n$ Latin square.
\end{theorem}

\section{A Doubling Construction for $t \equiv 4$ or $8~(\text{mod}~12)$}
\label{sec:doubling}

Let $n=2t$, where $t \equiv 4$ or $8~(\text{mod}~12)$, and let $\dS$ be a $\mathrm{BP}(t,4)$
on the point set $\Z_t$. Furthermore, let $\dT$ be a
resolvable $\mathrm{SQS}(t)$ and let $\cF=\{ \cF_0,\cF_1,\ldots,\cF_{t-2} \}$ be a one-factorization of $K_t$ on
the point set $\Z_t$. We form a $\mathrm{BP}(n,4)$ $\dP$ on the point set $\Z_t \times \Z_2$.
The parallel classes of $\dP$ will be of three types, Type~$\cS$, Type~$\cT$,
and Type $\cF$. A quadruple $\{ (x_1,i_1),(x_2,i_2),(x_3,i_3),(x_4,i_4)\}$ is
in a parallel class of Type~$\cS$ if $|\{x_1,x_2,x_3,x_4\}|=4$, in a parallel class of Type~$\cT$
if $|\{x_1,x_2,x_3,x_4\}|=3$, and in a parallel class of Type $\cF$ if $|\{x_1,x_2,x_3,x_4\}|=2$.
Each type does not have any quadruple with a different structure.

\subsection{Parallel Classes of Type $\cS$}
\label{sec:cS}

For a given parallel class $\cR$ of $\dS$ form eight sets $\cR_i$, $i \in \Z_8$, in $\dP$.
For each $i\in \Z_8$, given a block $\{ x_1,x_2,x_3,x_4 \} \in \cR$ we form the following two blocks in $\cR_i$.

$$
\{ (x_1,0),(x_2,j_2),(x_3,j_3),(x_4,j_4) \},~ \{ (x_1,1),(x_2,j_2 +1 ),(x_3,j_3+1),(x_4,j_4+1) \},
$$
where $(j_2 j_3 j_4)$ is the binary representation of $i$.

The following result can be easily verified.
\begin{lemma}
\label{lem:classS}
Given a parallel class $\cR$ of $\dS$, each set $\cR_i$, $i \in \Z_8$, is a parallel class,
of Type $\cS$, on the point set $\Z_t \times \Z_2$.
Each two parallel classes $\cR_i$ and $\cR'_j$, $0 \leq i,j \leq 7$, where $\cR \neq \cR'$ or $i \neq j$, are disjoint.
The number of parallel classes of Type $\cS$ in $\dP$ is $8 \binom{t-1}{3}$.
\end{lemma}
\begin{proof}
Since $\cR$ is a parallel class of $\dS$ on the point set $\Z_t$,
it follows that $\cR_i$ forms a parallel class on the point set $\Z_t \times \Z_2$ for each $i \in \Z_8$.
If $B=\{(x_1,\ell_1),(x_2,\ell_2),(x_3,\ell_3),(x_4,\ell_4)\}$ is a block of a defined parallel class $\cR_i$, then $\{x_1,x_2,x_3,x_4\}$ is
a block of the parallel class $\cR$. Hence, $|\{x_1,x_2,x_3,x_4\}|=4$ which implies that $B$ is a block in a parallel class of Type $\cS$.

Suppose that $\cR_i$ and $\cR'_j$ are two parallel classes which are not disjoint and let\linebreak
$B=\{(x_1,\ell_1),(x_2,\ell_2),(x_3,\ell_3),(x_4,\ell_4)\}\in  \cR_i\cap \cR'_j$.
By the definition of $\cR_i$ and $\cR'_j$, we have that $\{x_1,x_2,x_3,x_4\}\in \cR \cap \cR'$.
Since both $\cR$ and $\cR'$ are parallel classes of $\dS$, in which each two different parallel
classes are disjoint, it follows that $\cR = \cR'$.
Moreover, the value of $\ell_1$ implies that either $\ell_2 \ell_3 \ell_4$
or $(\ell_2+1) (\ell_3+1) (\ell_4+1)$ is the binary representation of $i$ and $j$ which
implies that $i=j$, that is $\cR_i = \cR'_j$, a contradiction. Thus, $\cR_i \cap \cR'_j = \varnothing$
if $\cR \neq \cR'$ or $i \neq j$.

By Theorem \ref{thm:num_classes}, the number of parallel classes in $\dS$ is $\binom{t-1}{3}$.
Each parallel class of $\dS$ induces eight parallel classes of Type $\cS$ in $\dP$.
Hence, the number of parallel classes of Type $\cS$ in $\dP$ is $8 \binom{t-1}{3}$.
\end{proof}

\begin{example}
\label{ex:doub8S}
For $t=8$ there are $35$ parallel classes in a unique $\mathrm{BP}(8,4)$, where a parallel class has two
quadruples of the form $\{ X, \Z_8 \setminus X \}$, where $X$ is a quadruple of $\Z_8$.
For the parallel class ${\cR = \{ \{0,1,2,3\},\{4,5,6,7\} \}}$ form the eight
parallel classes $\cR_0, \cR_1, \cR_2, \cR_3, \cR_4, \cR_5, \cR_6, \cR_7$ for $\dP$ as follows.

\begin{footnotesize}
$$
\cR_0 =\{\{(0,0),(1,0),(2,0),(3,0)\},\{(0,1),(1,1),(2,1),(3,1)\},\{(4,0),(5,0),(6,0),(7,0)\},\{(4,1),(5,1),(6,1),(7,1)\}\},
$$
$$
\cR_1 =\{\{(0,0),(1,0),(2,0),(3,1)\},\{(0,1),(1,1),(2,1),(3,0)\},\{(4,0),(5,0),(6,0),(7,1)\},\{(4,1),(5,1),(6,1),(7,0)\}\},
$$
$$
\cR_2 =\{\{(0,0),(1,0),(2,1),(3,0)\},\{(0,1),(1,1),(2,0),(3,1)\},\{(4,0),(5,0),(6,1),(7,0)\},\{(4,1),(5,1),(6,0),(7,1)\}\},
$$
$$
\cR_3 =\{\{(0,0),(1,0),(2,1),(3,1)\},\{(0,1),(1,1),(2,0),(3,0)\},\{(4,0),(5,0),(6,1),(7,1)\},\{(4,1),(5,1),(6,0),(7,0)\}\},
$$
$$
\cR_4 =\{\{(0,0),(1,1),(2,0),(3,0)\},\{(0,1),(1,0),(2,1),(3,1)\},\{(4,0),(5,1),(6,0),(7,0)\},\{(4,1),(5,0),(6,1),(7,1)\}\},
$$
$$
\cR_5 =\{\{(0,0),(1,1),(2,0),(3,1)\},\{(0,1),(1,0),(2,1),(3,0)\},\{(4,0),(5,1),(6,0),(7,1)\},\{(4,1),(5,0),(6,1),(7,0)\}\},
$$
$$
\cR_6 =\{\{(0,0),(1,1),(2,1),(3,0)\},\{(0,1),(1,0),(2,0),(3,1)\},\{(4,0),(5,1),(6,1),(7,0)\},\{(4,1),(5,0),(6,0),(7,1)\}\},
$$
$$
\cR_7 =\{\{(0,0),(1,1),(2,1),(3,1)\},\{(0,1),(1,0),(2,0),(3,0)\},\{(4,0),(5,1),(6,1),(7,1)\},\{(4,1),(5,0),(6,0),(7,0)\}\}.
$$
\end{footnotesize}
\end{example}

\subsection{Parallel Classes of Type $\cT$}
\label{sec:cT}

For a given parallel class $\cR$ of $\dT$ form 24 sets
$\cR_{i,j,k}$, $i \in \Z_6$,  $j,k \in \Z_2$, in $\dP$.
Given a block $\{ x_1,x_2,x_3,x_4 \} \in \cR$ we form the following two blocks in $\cR_{0,j,k}$:
$$
\{ (x_1,0),(x_1,1),(x_2,j),(x_3,k) \},~\{ (x_2,j+1),(x_3,k+1),(x_4,0),(x_4,1) \}.
$$
The following two blocks are constructed in $\cR_{1,j,k}$, $j,k \in \Z_2$:
$$
\{ (x_1,0),(x_1,1),(x_2,j),(x_4,k) \},~\{ (x_2,j+1),(x_3,0),(x_3,1),(x_4,k+1) \}.
$$
The following two blocks are constructed in $\cR_{2,j,k}$, $j,k \in \Z_2$:
$$
\{ (x_1,0),(x_1,1),(x_3,j),(x_4,k) \},~\{ (x_2,0),(x_2,1),(x_3,j+1),(x_4,k+1) \}.
$$
The following two blocks are constructed in $\cR_{3,j,k}$, $j,k \in \Z_2$:
$$
\{ (x_1,j),(x_2,0),(x_2,1),(x_3,k) \},~\{ (x_1,j+1),(x_3,k+1),(x_4,0),(x_4,1) \}.
$$
The following two blocks are constructed in $\cR_{4,j,k}$, $j,k \in \Z_2$:
$$
\{ (x_1,j),(x_2,0),(x_2,1),(x_4,k) \},~\{ (x_1,j+1),(x_3,0),(x_3,1),(x_4,k+1) \}.
$$
The following two blocks are constructed in $\cR_{5,j,k}$, $j,k \in \Z_2$:
$$
\{ (x_1,j),(x_2,k),(x_3,0),(x_3,1) \},~\{ (x_1,j+1),(x_2,k+1),(x_4,0),(x_4,1) \}.
$$

The following result can be easily verified.
\begin{lemma}
\label{lem:classT}
Given a parallel class $\cR$ of $\dT$, each set $\cR_{i,j,k}$, $i \in \Z_6$, $j,k \in \Z_2$,
is a parallel class, of Type $\cT$, on the point set $\Z_t \times \Z_2$.
Each two parallel classes $\cR_{i,j,k}$ and $\cR'_{m,r,s}$, $i,m \in Z_6$, $j,k,r,s \in \Z_2$,
where $\cR \neq \cR'$ or $(i,j,k) \neq (m,r,s)$, are disjoint.
The number of parallel classes of Type $\cT$ in $\dP$ is $4(t-1)(t-2)$.
\end{lemma}
\begin{proof}
It is easy to verify that since $\cR$ is a parallel class of $\dT$ on the point set $\Z_t$,
it follows that $\cR_{i,j,k}$ forms a parallel class on the point set $\Z_t \times \Z_2$
for each $i \in \Z_6$, $j,k \in \Z_2$.
Moreover, if $B=\{(x_1,\ell_1),(x_2,\ell_2),(x_3,\ell_3),(x_4,\ell_4)\}$ is a block of a defined parallel class,
one can easily verified that $|\{x_1,x_2,x_3,x_4\}|=3$
which implies that $B$ is a block in a parallel class of Type $\cT$.

Suppose that $\cR_{i,j,k}$ and $\cR'_{m,r,s}$ are two parallel classes which are not disjoint
and\linebreak $B=\{(x_1,\ell_1),(x_2,\ell_2),(x_3,\ell_3),(x_4,\ell_4)\}\in  \cR_{i,j,k} \cap \cR'_{m,r,s}$.
Since $|\{x_1,x_2,x_3,x_4\}|=3$ and $\dT$ is an $\mathrm{SQS}(t)$,
one can easily verify from the construction that $B \in  \cR_{i,j,k} \cap \cR'_{m,r,s}$ can only happen if $i=m$.
Assume, for example, that $i=0$, $\{y_1,y_2,y_3,y_4\} \in \dT$, and
$B=\{(y_1,0),(y_1,1),(y_2,\alpha),(y_3,\beta)\}\in  \cR_{0,j,k} \cap \cR'_{0,r,s}$.
It implies that $(\alpha,\beta)=(j,k)=(r,s)$. Moreover, it implies that $\{y_1,y_2,y_3,y_4\} \in \cR \cap \cR'$
and hence $\cR=\cR'$, that is $\cR_{i,j,k} = \cR'_{m,r,s}$, a contradiction.
Thus, $\cR_{i,j,k} \cap \cR'_{m,r,s} = \varnothing$ if $\cR \neq \cR'$ or $(i,j,k) \neq (m,r,s)$.

Simple counting arguments show that the number of parallel classes in an $\mathrm{SQS}(t)$ is $\frac{(t-1)(t-2)}{6}$.
Each parallel class of $\dT$ in an $\mathrm{SQS}(t)$ induces 24 parallel classes of Type $\cT$ in $\dP$.
Hence, the number of parallel classes of Type $\cT$ in $\dP$ is $4(t-1)(t-2)$.
\end{proof}

\subsection{Parallel Classes of Type $\cF$}
\label{sec:cF}

For each one-factor $\cF_i$, $0 \leq i \leq t-2$,
form the following set $\hat{\cR}_i$ in $\dP$.

$$
\hat{\cR}_i \triangleq \{ \{ (x,0),(y,0),(x,1),(y,1) \} ~:~ \{x,y\} \in \cF_i \}
$$

The following result can be easily verified.
\begin{lemma}
\label{lem:classF}
For each one-factor $\cF_i$, $0 \leq i \leq t-2$, the set $\hat{\cR}_i$ is a parallel class
of size $t/2$ on the point set $\Z_t \times \Z_2$.
Each two parallel classes $\hat{\cR}_i$ and $\hat{\cR}_j$, $0 \leq i < j \leq t-2$, are disjoint.
Each defined block $\{ (x_1,i_1),(x_2,i_2),(x_3,i_3),(x_4,i_4)\}$ of a
parallel class $\hat{\cR}$ of Type $\cF$ satisfies ${|\{x_1,x_2,x_3,x_4\}|=2}$.
The number of parallel classes of Type $\cF$ in $\dP$ is $t-1$.
\end{lemma}

\subsection{Conclusion for $t \equiv 4$ or $8~(\text{mod}~12)$}

The constructions of the blocks of Type $\cS$, Type $\cT$, and Type $\cF$ imply the following theorem.
\begin{theorem}
\label{thm:4_8m12}
If there exists a $\mathrm{BP}(t,4)$, where  $t \equiv 4$ or $8~(\text{mod}~12)$, then there exists a $\mathrm{BP}(2t,4)$.
\end{theorem}
\begin{proof}
By Lemmas~\ref{lem:classS},~\ref{lem:classT}, and~\ref{lem:classF}, all the defined
parallel classes of Type $\cS$, Type $\cT$, and Type $\cF$, are pairwise disjoint and there are
$$
8 \binom{t-1}{3} + 4(t-1)(t-2) + (t-1) = \binom{2t-1}{3}
$$
such parallel classes. This implies by Theorem~\ref{thm:num_classes} that each quadruple of $\Z_t \times \Z_4$ is
contained in exactly one of the parallel classes.
\end{proof}

\section{Ingredients for the Quadrupling Construction}
\label{sec:general}

Let $n=4t$, where $t$ is a positive integer, $t+\delta_1$ is divisible
by 4, and $t+ \delta_2$ is divisible by~3, where $0 \leq \delta_1 \leq 3$, and $ 0 \leq \delta_2 \leq 2$.
Assume that $\dS$ is a $\mathrm{BP}(t+\delta_1,4)$ on the point set $\Z_t \cup \{ \Omega_1,\ldots,\Omega_{\delta_1} \}$.
Assume $\dT$ is a $\mathrm{BP}(t+\delta_2,3)$ on the point set $\Z_t \cup \{ \Omega_1,\ldots,\Omega_{\delta_2} \}$.
The constructed $\mathrm{BP}(4t,4)$, $\dP$, in the following sections is on the point set $\Z_t \times \Z_4$.
In a quadruple $\{ y_0 , y_1 ,y_2 , y_3 \} \in \dS$ and a triple $\{ x_0,x_1,x_2 \} \in \dT$,
where $y_0 , y_1 ,y_2 , y_3 , x_0,x_1,x_2 \in \Z_t$ we
assume that the elements are ordered in increasing order, i.e. $y_0<y_1<y_2<y_3$ and $x_0<x_1<x_2$.
The number of quadruples in $\Z_t \times \Z_4$ is $\binom{4t}{4}$, the number of quadruples in a parallel class
of $\dP$ is $t$, and the number of parallel classes in $\dP$ is $\binom{4t-1}{3}$.

\subsection{Configurations, Groups, and Types of Parallel Classes}
\label{sec:config}

Each quadruple $X$ of the point set $\Z_t \times \Z_4$ is written as
$X= \{ (x_0,i_0), (x_1,i_1),(x_2,i_2),(x_3,i_3) \}$. Let $j_0 = |X \cap \Z_t \times \{0\}|$,
$j_1 = |X \cap \Z_t \times \{1\}|$, $j_2 = |X \cap \Z_t \times \{2\}|$, and
$j_3 = |X \cap \Z_t \times \{3\}|$. We say that $X$ is a quadruple from
configuration $(j_0,j_1,j_2,j_3)$. There is a total of 35 possible configurations.
For $j_i$, $i \in \Z_4$, there is a configuration $(j_0,j_1,j_2,j_3)$ if $j_0+j_1+j_2+j_3=4$.
There are $\prod_{i=0}^3 \binom{t}{j_i}$ distinct quadruples in configuration $(j_0,j_1,j_2,j_3)$.
The 35 configurations are partitioned into five groups as follows:

\noindent
{\bf Group 1:} In this group there are four configurations $(4,0,0,0)$, $(0,4,0,0)$,
$(0,0,4,0)$, and $(0,0,0,4)$.

\noindent
{\bf Group 2:} In this group there are twelve configurations $(3,1,0,0)$, $(1,3,0,0)$, $(3,0,1,0)$, $(1,0,3,0)$, $(3,0,0,1)$,
$(1,0,0,3)$, $(0,3,1,0)$, $(0,1,3,0)$, $(0,3,0,1)$, $(0,1,0,3)$, $(0,0,3,1)$, and $(0,0,1,3)$.

\noindent
{\bf Group 3:} In this group there are six configurations $(2,2,0,0)$, $(2,0,2,0)$, $(2,0,0,2)$,
$(0,2,2,0)$, $(0,2,0,2)$, and $(0,0,2,2)$.

\noindent
{\bf Group 4:} In this group there are twelve configurations $(2,1,1,0)$, $(2,1,0,1)$, $(2,0,1,1)$, $(1,2,1,0)$, $(1,2,0,1)$,
$(0,2,1,1)$, $(1,1,2,0)$, $(1,0,2,1)$, $(0,1,2,1)$, $(1,1,0,2)$, $(1,0,1,2)$, and $(0,1,1,2)$.

\noindent
{\bf Group 5:} The only configuration contained in this group is $(1,1,1,1)$.

The construction which will be described in the sequel has five types of parallel classes.
\begin{itemize}
\item In Type 1, most of the quadruples in each parallel class are from Group 1.
There might be some quadruples from Group 2, Group 3, or Group 5.

\item In Type 2, if $t$ is divisible by 3 then all the quadruples in each parallel class are from Group~2.
If $t$ is not divisible by 3, then there might be some quadruples from Group 3 or Group 4.

\item In Type 3, if $t$ is even then all the quadruples in each parallel class are from Group 3.
If $t$ is odd, then $t-1$ of the quadruples in each parallel class are from Group 3 and
one quadruple in each parallel class is from Group 5, i.e. from configuration (1,1,1,1).

\item In Type 4, if $t$ is even then all the quadruples in each parallel class are from Group 4.
If $t$ is odd then $t-1$ of the quadruples in each parallel class are from Group 4 and
one quadruple in each parallel class is from Group 5, i.e. from configuration (1,1,1,1).

\item In Type 5, all the quadruples in each parallel class are from Group 5.
\end{itemize}

\section{Parallel Classes for the Quadrupling Construction}
\label{sec:const_par}

\subsection{Parallel Classes of Type 1}
\label{sec:T4}

Given any parallel class $\cR$ in $\dS$, construct a parallel class $\cR'$ in $\dP$.
For any given $B \in \cR$ we form a subset of quadruples in $\cR'$ which will be defined in the sequel.
We distinguish between the four different residues of $t$ modulo 4.

\subsubsection{Type 1 for $t$ divisible by 4}
\label{sec:T4t0}

The set $\dS$ is a $\mathrm{BP}(t,4)$ on $\Z_t$ and $\cR$ a parallel class in $\dS$.
If $B= \{ x_0,x_1,x_2, x_3 \} \in \cR$, then construct the following set with four quadruples in $\cR'$.
$$
B'_i \triangleq \{ \{ (x_0,i),(x_1,i),(x_2,i),(x_3,i) \} ~:~ i \in \Z_4 \}.
$$

\begin{lemma}
\label{lem:T4t0}
$~$
\begin{enumerate}
\item The subset $\cR'$ is a parallel class on the point set $\Z_t \times \Z_4$.

\item Each quadruple of Group 1 is contained in exactly one such parallel class.

\item The number of parallel classes of Type 1 when $t$ is divisible by 4 is $\binom{t-1}{3}$.
Each such parallel class contains $t$ quadruples from Group 1.
\end{enumerate}
\end{lemma}
\begin{proof}
Since $\cR$ is a parallel class in $\dS$, it follows that
$$
\cR'=\bigcup_{B\in \cR, i \in \Z_4} B'_i
$$
forms a parallel class on the point set $\Z_t \times \Z_4$.

For each $i \in \Z_4$, each quadruple $\{(x_0,i),(x_1,i),(x_2,i),(x_3,i)\}$ is contained in exactly
one parallel class in $\dP$ since any block $\{x_0,x_1,x_2,x_3\}$ is contained in exactly one parallel class in $\dS$.

By Theorem~\ref{thm:num_classes}, the number of parallel classes in a $\mathrm{BP}(t,4)$ is $\binom{t-1}{3}$.
Each parallel class of $\dS$ corresponds to exactly one parallel class of Type 1 in $\dP$, and
hence the number of parallel classes of Type 1 in $\dP$ is also $\binom{t-1}{3}$.
Clearly, each such parallel class contains $t$ quadruples from Group~1.
\end{proof}

\subsubsection{Type 1 for $t \equiv 3~(\text{mod}~4)$}
\label{sec:T4t3}

The set $\dS$ is a $\mathrm{BP}(t+1,4)$ on the point set $\Z_t \cup \{ \Omega_1 \}$ and $\cR$ is a parallel class in $\dS$.
For each block $B \in \cR$ construct a set of quadruples $\cR'$ as follows.

\begin{enumerate}
\item
If $\Omega_1 \notin B=\{x_0,x_1,x_2,x_3\}$ then construct the following four quadruples in $\cR'$:
$$
\{ (x_0,i),(x_1,i),(x_2,i),(x_3,i)\}, ~~ i \in \Z_4.
$$

\item If $B= \{ x_0,x_1,x_2,\Omega_1 \}$, then construct the following three quadruples in $\cR'$:
$$
\{ (x_0,i),(x_1,i),(x_2,i),(x_i,3) \}, ~~ i \in \{0,1,2 \} .
$$
\end{enumerate}

\begin{lemma}
\label{lem:rm31f4}
The subset of quadruples from Group 2 contained in the parallel classes of Type 1 when $t \equiv 3~(\text{mod}~4)$ is
$$
\{ \{ (x,0),(y,0),(z,0),(x,3) \}, \{ (x,1),(y,1),(z,1),(y,3) \}, \{ (x,2),(y,2),(z,2),(z,3) \} ~:~ x,y,z \in \Z_t \} .
$$
This subset contains exactly $(t-2) \binom{t}{2}$ quadruples from Group 2.
\end{lemma}
\begin{proof}
In a $\mathrm{BP}(t+1,4)$, any four different points from $\Z_t \cup \{ \Omega_1 \}$ occurs as a block exactly once.
Hence, the point $\Omega_1$ occurs with any three points, of $\Z_t$, in some block exactly once.
Therefore, the number of quadruples containing $\Omega_1$ is $\binom{t}{3}$. Each such quadruple
induces three quadruples from Group~2 defined above.
Thus, this set contains exactly $3 \binom{t}{3}=(t-2) \binom{t}{2}$ quadruples from Group 2.
\end{proof}

\begin{lemma}
\label{lem:T4t3}
$~$
\begin{enumerate}
\item The subset $\cR'$ is a parallel class on the point set $\Z_t \times \Z_4$.

\item Each quadruple of Group 1 is contained in exactly one such parallel class.

\item The number of parallel classes of Type 1 when $t \equiv 3~(\text{mod}~4)$ is $\binom{t}{3}$.
Each such parallel class contains $t-3$ quadruples from Group 1 and three quadruples from Group 2.
\end{enumerate}
\end{lemma}
\begin{proof}
The first two claims of the lemma are readily verified.
For the third claim, note that each parallel class of $\dS$ corresponds to a unique parallel class of $\dP$.
By Theorem~\ref{thm:num_classes} there are $\binom{t}{3}$ parallel classes in a $\mathrm{BP}(t+1,4)$ and hence the number of parallel
classes of Type 1, in $\dP$ when $t \equiv 3~(\text{mod}~4)$, is $\binom{t}{3}$.
Clearly, each such parallel class contains $t-3$ quadruples from Group 1 and three quadruples from Group 2.
\end{proof}

\subsubsection{Type 1 for $t \equiv 2~(\text{mod}~4)$}
\label{sec:T4t2}

The set $\dS$ is a $\mathrm{BP}(t+2,4)$ on the point set $\Z_t \cup \{ \Omega_1, \Omega_2 \}$ and $\cR$ is a parallel class in $\dS$.
For each block $B \in \cR$ define the following blocks in $\cR'$.

\begin{enumerate}
\item
If $\Omega_1, \Omega_2 \notin B=\{x_0,x_1,x_2,x_3\}$ then construct the following four quadruples in $\cR'$.
$$
\{ (x_0,i),(x_1,i),(x_2,i),(x_3,i)\}, ~~ i \in \Z_4.
$$

\item If $\Omega_2 \notin B= \{ x_0,x_1,x_2,\Omega_1 \}$, then construct the following three quadruples in $\cR'$.
$$
\{ (x_0,i),(x_1,i),(x_2,i),(x_i,3) \}, ~~i \in \Z_3.
$$

\item If $\Omega_1 \notin B= \{ x_0,x_1,x_2,\Omega_2 \}$, then construct the following three quadruples in $\cR'$.
$$
\{ (x_0,i),(x_1,i),(x_2,i),(x_{i+1},3) \}, ~~i \in \Z_3.
$$

\item If $B= \{ x_0,x_1,\Omega_1,\Omega_2 \}$, then construct the two quadruples in $\cR'$.
$$
\{ (x_0,0),(x_1,0),(x_0,1),(x_1,1) \} \}.
$$
$$
\{ (x_0,2),(x_1,2),(x_0,3),(x_1,3) \} \}.
$$
\end{enumerate}

\vspace{0.3cm}

Recall, that there is an order between the elements of each block of $\dS$ which implies that there
is no ambiguity in the definitions of the blocks of $\dP$.

\vspace{0.5cm}

\begin{lemma}
\label{lem:rm22a31f4}
$~$
\begin{enumerate}
\item The subset of quadruples from Group 2 contained in the parallel classes of Type 1 when $t \equiv 2~(\text{mod}~4)$ is
{\small
$$
\{ \{ (x,0),(y,0),(z,0),(x,3) \}, \{ (x,1),(y,1),(z,1),(y,3) \}, \{ (x,2),(y,2),(z,2),(z,3) \} ~:~ x,y,z \in \Z_t \}
$$
$$
\cup \{ \{ (x,0),(y,0),(z,0),(y,3) \}, \{ (x,1),(y,1),(z,1),(z,3) \}, \{ (x,2),(y,2),(z,2),(x,3) \} ~:~ x,y,z \in \Z_t \} .
$$
}
This subset contains exactly $t(t-1)(t-2)$ quadruples from Group 2.

\item The subset of quadruples from Group 3 contained in the parallel classes of Type 1 when $t \equiv 2~(\text{mod}~4)$ is
$$
\{ \{ (x,i),(y,i),(x,i+1),(y,i+1) \} ~:~ x,y \in \Z_t ,~ i=0,2 \}  ~.
$$
This subset contains exactly $(t-1)t$ quadruples from Group 3.
\end{enumerate}
\end{lemma}
\begin{proof}
The structure of the quadruples in $\dP$ when $t \equiv 2~(\text{mod}~4)$ is an immediate consequence
from the definition of the blocks in the construction.

It is easily to verify that $\Omega_1$ (without $\Omega_2$) is contained $\binom{t}{3}$ quadruples in $\dS$ and
the same is true for $\Omega_2$ (without $\Omega_1$).
Hence, there are exactly $2 \cdot 3\cdot \binom{t}{3}=t(t-1)(t-2)$ quadruples from Group 2 in $\dP$.

Since the pair $\{\Omega_1,\Omega_2\}$ is contained in $\binom{t}{2}$ quadruples of $\dS$,
it follows that there are exactly $2\binom{t}{2}=(t-1)t$ quadruples from Group 3 in $\dP$.
\end{proof}

\begin{lemma}
\label{lem:T4t2}
$~$
\begin{enumerate}
\item The subset $\cR'$ is a parallel class, of Type 1, on the point set $\Z_t \times \Z_4$.

\item Each quadruple of Group 1 is contained in exactly one such parallel class.

\item The number of parallel classes of Type 1 when $t \equiv 2~(\text{mod}~4)$ is $\binom{t+1}{3}$.
There are $\binom{t}{2}$ parallel classes, in which each parallel class contains $t-2$ quadruples
from Group 1 and two quadruples from Group 3.
There are $\binom{t}{3}$ parallel classes, in which each parallel class contains $t-6$ quadruples
from Group 1 and six quadruples from Group 2.
\end{enumerate}
\end{lemma}
\begin{proof}
The first two claims of the lemma are readily verified.
For the third claim, note that $\Omega_1$ ($\Omega_2$, respectively) without $\Omega_2$ ($\Omega_1$, respectively) occurs
in exactly $\binom{t}{3}$ parallel classes of $\dS$. Each such parallel class of $\dS$ (which contains one block
with $\Omega_1$ and another block with $\Omega_2$) induces
one parallel class of Type 1 in $\dP$, which contains $t-6$ quadruples
from Group 1 and six quadruples from Group 2.
Similarly, the pair $\{\Omega_1, \Omega_2\}$ occurs in exactly $\binom{t}{2}$ parallel classes of $\dS$.
Each such parallel class induces one parallel class of Type 1 in $\dP$, which contains $t-2$ quadruples
from Group 1 and two quadruples from Group 3.
Thus, the total number of parallel classes of Type 1 is $\binom{t}{3}+\binom{t}{2}=\binom{t+1}{3}$.
\end{proof}

\subsubsection{Type 1 for $t \equiv 1~(\text{mod}~4)$}
\label{sec:T4t1}

The set $\dS$ is a $\mathrm{BP}(t+3,4)$ on the point set $\Z_t \cup \{ \Omega_1, \Omega_2 , \Omega_3 \}$
and $\cR$ is a parallel class in $\dS$. For each block $B \in \cR$ construct the following blocks in $\cR' \in \dP$.

\begin{enumerate}
\item
If $\Omega_1, \Omega_2 , \Omega_3 \notin B=\{x_0,x_1,x_2,x_3\}$ then construct the following four quadruples in $\cR'$.
$$
\{ (x_0,i),(x_1,i),(x_2,i),(x_3,i)\}, ~~ i \in \Z_4.
$$

\item If $\Omega_1, \Omega_2 \notin B= \{ x_0,x_1,x_2,\Omega_3 \}$, then construct the following three quadruples in $\cR'$.
$$
\{ (x_0,i),(x_1,i),(x_2,i),(x_i,3) \}, ~~i \in \Z_3.
$$

\item If $\Omega_1, \Omega_3 \notin B= \{ x_0,x_1,x_2,\Omega_2 \}$, then construct the following three quadruples in $\cR'$.
$$
\{ (x_0,i),(x_1,i),(x_2,i),(x_{i+1},3) \}, ~~i \in \Z_3.
$$

\item If $\Omega_2, \Omega_3 \notin B= \{ x_0,x_1,x_2,\Omega_1 \}$, then construct the following three quadruples in $\cR'$.
$$
\{ (x_0,i),(x_1,i),(x_2,i),(x_{i+2},3) \}, ~~i \in \Z_3.
$$

\item If $\Omega_1 \notin B= \{ x_0,x_1, \Omega_2 ,\Omega_3 \}$, then construct the following two quadruples in $\cR'$.
$$
\{ (x_0,0),(x_1,0),(x_0,1),(x_1,1) \},
$$
$$
\{ (x_0,2),(x_1,2),(x_0,3),(x_1,3) \}.
$$

\item If $\Omega_2 \notin B= \{ x_0,x_1, \Omega_1 ,\Omega_3 \}$, then construct the following two quadruples in $\cR'$.
$$
\{ (x_0,0),(x_1,0),(x_0,3),(x_1,3) \},
$$
$$
\{ (x_0,1),(x_1,1),(x_0,2),(x_1,2) \}.
$$

\item If $\Omega_3 \notin B= \{ x_0,x_1, \Omega_1 ,\Omega_2 \}$, then construct the following two quadruples in $\cR'$.
$$
\{ (x_0,0),(x_1,0),(x_0,2),(x_1,2) \},
$$
$$
\{ (x_0,1),(x_1,1),(x_0,3),(x_1,3) \}.
$$

\item If $B= \{ x_0,\Omega_1,\Omega_2 , \Omega_3 \}$, then construct the following quadruple in $\cR'$.
$$
\{ (x_0,0),(x_0,1),(x_0,2),(x_0,3) \} \}.
$$
\end{enumerate}

\begin{lemma}
\label{lem:rm22a31a1111f4}
$~$
\begin{enumerate}
\item The subset of quadruples from Group 2 contained in the parallel classes of Type 1 when $t \equiv 1~(\text{mod}~4)$ is
{\small
$$
\{ \{ (x,i),(y,i),(z,i),(x,3) \}, \{ (x,i),(y,i),(z,i),(y,3) \}, \{ (x,i),(y,i),(z,i),(z,3) \} : x,y,z \in \Z_t,~ i=0,1,2 \} .
$$
}
This subset contains exactly $9 \binom{t}{3}$ quadruples from Group 2.

\item The subset of quadruples from Group 3 contained in parallel classes of Type 1 when $t \equiv 1~(\text{mod}~4)$ is
$$
\{ \{ (x,0),(y,0),(x,1),(y,1) \} ~:~ x,y \in \Z_t \}  \cup \{ \{ (x,2),(y,2),(x,3),(y,3) \} ~:~ x,y \in \Z_t \}
$$
$$
\cup \{ \{ (x,0),(y,0),(x,2),(y,2) \} ~:~ x,y \in \Z_t \}  \cup \{ \{ (x,1),(y,1),(x,3),(y,3) \} ~:~ x,y \in \Z_t \}
$$
$$
\cup \{ \{ (x,0),(y,0),(x,3),(y,3) \} ~:~ x,y \in \Z_t \}  \cup \{ \{ (x,1),(y,1),(x,2),(y,2) \} ~:~ x,y \in \Z_t \}~.
$$
This subset contains exactly $3(t-1)t$ quadruples from Group 3.

\item The subset of quadruples from Group 5 contained in the parallel classes of Type 1 when $t \equiv 1~(\text{mod}~4)$ is
$$
\{ \{ (i,0),(i,1),(i,2),(i,3) \} ~:~ i \in \Z_t \} ~.
$$
This subset contains exactly $t$ quadruples from Group 5 which form a parallel class.
\end{enumerate}
\end{lemma}
\begin{proof}
$~$
\begin{enumerate}
\item
When $t \equiv 1~(\text{mod}~4)$, for each $i\in \Z_3$ and $B= \{ x,y,z,\Omega_i \} \in \cR$,
the subset of quadruples from Group 2 contained in the parallel classes of Type 1 contains the following set with three quadruples
$$
\{ \{ (x,i),(y,i),(z,i),(x,3) \}, \{ (x,i),(y,i),(z,i),(y,3) \}, \{ (x,i),(y,i),(z,i),(z,3) \}\}
$$
In addition, for each $i\in \Z_3$, $\Omega_1$ (without $\Omega_2$ or $\Omega_3$) is contained in $\binom{t}{3}$ quadruples of $\dS$.
The same is true for $\Omega_2$ (without $\Omega_1$ or $\Omega_3$) and $\Omega_3$ (without $\Omega_1$ or $\Omega_2$).
Hence, there are exactly $3\cdot 3\cdot \binom{t}{3}=9\binom{t}{3}$ quadruples from Group 2 in $\dP$.

\item
For any $1\leq i<j\leq 3$ and $B= \{ x,y,\Omega_i,\Omega_j\} \in \cR$, the subset of quadruples from Group 3
which are used in Type 1 contains two of the quadruples defined above.
Since the pair $\{\Omega_i,\Omega_j\}$ is contained in $\binom{t}{2}$ quadruples of $\dS$,
there are exactly $3\cdot 2\cdot \binom{t}{2}=3(t-1)t$ quadruples from Group~3 in $\dP$.

\item
For any quadruple $\{x, \Omega_1, \Omega_2,\Omega_3\} \in \cR$ in $\dS$, the corresponding quadruple from Group 5
contained in the parallel classes of Type 1 is $\{(x,0),(x,1),(x,2),(x,3)\}$.
Since the triple $\{\Omega_1, \Omega_2,\Omega_3\}$ is contained in $t$ quadruples of $\dS$, it follows
that there are exactly $t$ quadruples from Group 5 in $\dP$.
\end{enumerate}
\end{proof}

The claims in the following lemma are proved similarly to the ones in previous lemmas and
hence we omit the proof and leave it to the interested reader.
\begin{lemma}
\label{lem:T4t1}
$~$
\begin{enumerate}
\item The subset $\cR'$ is a parallel class on the point set $\Z_t \times \Z_4$.

\item Each quadruple of Group 1 is contained in exactly one such parallel class.

\item The number of parallel classes of Type 1 when $t \equiv 1~(\text{mod}~4)$ is $\binom{t+2}{3}$.
There are $t$ parallel classes, where each parallel class contains $t-1$ quadruples
from Group 1 and one quadruple from Group~5.
There are $3\binom{t}{2}$ parallel classes, in which each parallel class contains $t-5$ quadruples
from Group 1, three quadruples of Group 2, and two quadruples from Group 3.
There are $\binom{t}{3} - \binom{t}{2}$ parallel classes, in which each parallel class contains $t-9$ quadruples
from Group 1 and nine quadruples from Group 2.
\end{enumerate}
\end{lemma}

\subsection{Parallel Classes of Type 2 for $t$ divisible by 3}
\label{sec:T31}

Given any parallel class $\cR$ in $\dT$, a few parallel classes are defined for $\dP$.
The construction heavily depends on the residue of $t$ modulo 4.

Let $M$ be a $t \times t$ Latin square on the point set $\Z_t$,
where the rows and the columns are indexed by~$\Z_t$.
Given any parallel class $\cR$ in $\dT$, construct $4t$, $4t-1$, $4t-2$, or $4t-3$, parallel classes in $\dP$,
depending if $t~(\text{mod}~4)$ equals 0, 3, 2, or 1, respectively.
The first $3t$ parallel classes will be denoted by $\cR_{i,j}$, $i \in \Z_t$,
$j \in \{0,1,2\}$. For a given $j \in \{0,1,2\}$, let $\phi : \{0,1,2,3\} \setminus \{j\} \longrightarrow \{0,1,2\}$ be any bijection.
Let $B_0, B_1,\ldots,B_{(t-3)/3}$, be the blocks of $\cR$.
For a given $i \in \Z_t$, $j \in \{ 0,1,2 \}$, define the following $t$ quadruples in~$\cR_{i,j}$.
$$
\{ (x_0,s),(x_1,s),(x_2,s),(M(i,3m+\phi(s)),j)  \}, ~\{ x_0,x_1,x_2\} \in B_m, ~ 0 \leq m \leq \frac{t-3}{3}, ~s \in \{0,1,2,3\} \setminus \{j\}.
$$

For the last $t$, $t-1$, $t-2$, $t-3$ parallel classes,
we distinguish now between four cases which have similar solutions, where we define
0, 1, 2, or 3, respectively, rows of the following matrix $M$.
If $t \equiv 1$ or $2$ or $3~(\text{mod}~4)$ and
if $B_m = \{ x_{m_0},x_{m_1},x_{m_2}\}$, $0 \leq m \leq \frac{t-3}{3}$, then $M(0,3m+i)=x_{m_i}$, for $i \in \Z_3$.
If $t \equiv 1$ or $2~(\text{mod}~4)$ and
if $B_m = \{ x_{m_0},x_{m_1},x_{m_2}\}$, $0 \leq m \leq \frac{t-3}{3}$, then $M(1,3m+i)=x_{m_{i+1}}$, for $i \in \Z_3$.
If $t \equiv 1 ~(\text{mod}~4)$ and
if $B_m = \{ x_{m_0},x_{m_1},x_{m_2}\}$, $0 \leq m \leq \frac{t-3}{3}$, then $M(2,3m+i)=x_{m_{i+2}}$, for $i \in \Z_3$.
For each congruence, we complete $M$ to be a Latin square (see Theorem~\ref{thm:completLS}).
Each one of the rows of $M$, which was not predetermined corresponds to a parallel class
denoted by $\cR_i$, $i \in \Z_t$, $i \geq 0$ if $t$ is divisible by 4;
$i \geq 1$ if $t \equiv 3 ~(\text{mod}~4)$;
$i \geq 2$ if $t \equiv 2 ~(\text{mod}~4)$;
and $i \geq 3$ if $t \equiv 1 ~(\text{mod}~4)$.
Each parallel class has $t$ quadruples as follows.
$$
\{ (x_0,s),(x_1,s),(x_2,s),(M(i,3m+s),3)  \},  ~\{ x_0,x_1,x_2\} \in B_m, ~ 0 \leq m \leq \frac{t-3}{3}, ~ s \in \{0,1,2\}.
$$

\begin{remark}
We can avoid Theorem~\ref{thm:completLS} and define the completion of $M$ to
a $t \times t$ Latin square. We have not done this to restrict the length of this section.
\end{remark}

\begin{lemma}
\label{lem:T31}
Each set ($\cR_{i,j}$ or $\cR_i$) of Type 2 defined for $t$ divisible by 3 is a parallel class
on the point set $\Z_t \times \Z_4$ in $\dP$.
Moreover,
\begin{enumerate}
\item If $t$ is divisible by 4, then there are $4t$ parallel classes of Type 2. These parallel classes
contain all the quadruples of Group 2. There are $2t(t-1)(t-2) = 4t\binom{t-1}{2}$ such parallel classes,
in which each parallel class contains $t$ quadruples from Group~2.

\item If $t \equiv 1~(\text{mod}~4)$, then there are $4t-3$ parallel classes of Type 2. Each quadruple of
Group~2 is contained in either one of these parallel classes or in the parallel classes of
Type 1. There are $(4t-3)\binom{t-1}{2}$ such parallel classes,
in which each parallel class contains $t$ quadruples from Group~2.

\item If $t \equiv 2~(\text{mod}~4)$, then there are $4t-2$ parallel classes of Type 2. Each quadruple of
Group~2 is contained in either one of these parallel classes or in the parallel classes of
Type 1. There are $(2t-1)(t-1)(t-2)=(4t-2)\binom{t-1}{2}$ such parallel classes,
in which each parallel class contains $t$ quadruples from Group~2.

\item If $t \equiv 3~(\text{mod}~4)$, then there are $4t-1$ parallel classes of Type 2. Each quadruple of
Group~2 is contained in either one of these parallel classes or in the parallel classes of
Type 1. There are $(4t-1)\binom{t-1}{2}$ such parallel classes,
in which each parallel class contains $t$ quadruples from Group~2.
\end{enumerate}
Each such parallel class of Type 2 contains only quadruples from Group 2.
\end{lemma}
\begin{proof}
If $t$ is divisible by 4, then by Lemma~\ref{lem:T4t0}, no quadruples of Group 2 are contained in the parallel classes of Type 1.
Hence, for each parallel class of $\mathrm{BP}(t,3)$, there are $4t$ parallel classes of Type 2 in $\dP$. By Theorem~\ref{thm:num_classes},
there exist $\binom{t-1}{2}$ parallel classes in a $\mathrm{BP}(t,3)$.
Thus, there are $4t\binom{t-1}{2}$ such parallel classes in $\dP$, where each parallel class contains $t$ quadruples from Group 2.

Claims 2, 3, and 4 are proved in a similar way.
\end{proof}

\subsection{Parallel Classes of Type 3}
\label{sec:T22}

For the definition of the parallel classes of Type 3 in $\dP$ we
distinguish between even $t$ and odd $t$.

\subsubsection{Type 3 for even $t$}
\label{sec:T22even}

For even $t$ a one-factorization $\cF =\{ \cF_0,\cF_1,\ldots,\cF_{t-2} \}$ of $K_t$ is required.
For each $0 \leq i,s \leq t-2$, $0 \leq r \leq \frac{t-2}{2}$, and $a,b,c,d \in \Z_4$, where $|\{a,b,c,d\}|=4$,
the following set is defined in $\dP$:
{\small
$$
\{  \{ (x ,a),(y,a),(z,b),(v,b) \},\{ (x,c),(y,c),(z,d),(v,d) \}  ~:~
\{ x,y \} = \cF_{i,j}, ~\{z,v\} = \cF_{s,j+r}, ~ 0 \leq j \leq (t-2)/2  \}~.
$$
}

Note that there are three options for $a,b$, i.e.,
$\{a,b\}=\{0,1\}$, $\{a,b\}=\{0,2\}$ and $\{a,b\}=\{0,3\}$ and
$\{c,d\}=\{0,1,2,3\}\setminus\{a,b\}$.

\begin{lemma}
\label{lem:T22even}
$~$
\begin{enumerate}
\item For each $0 \leq i,s \leq t-2$, $0 \leq r \leq \frac{t-2}{2}$, and $a,b,c,d \in \Z_4$, where $|\{a,b,c,d\}|=4$,
the $t$ defined quadruples form a parallel class with $t$ quadruples from Group 3.

\item If $t$ is divisible by 4, then each quadruple from the configurations of Group 3 appears only in one
of the defined parallel classes.

\item If $t \equiv 2~(\text{mod}~4)$, then each quadruple from the configurations of Group 3 appears either in one
of the defined parallel classes or in the parallel classes of Type 1.

\item The number of defined parallel classes of Type 3 for even $t$ is $3(t-1) \binom{t}{2}$.
\end{enumerate}
\end{lemma}
\begin{proof}
Combining Lemma~\ref{lem:T4t0} and Lemma~\ref{lem:T4t2}, the first three claims are readily
verified from the definition of the blocks in the construction.
Since, in the definition of the construction, there are three options for $a,b,c,d$ and for each such option the triple $(i,s,r)$
has $(t-1)^2 \frac{t}{2}$ different choices, it follows that there are
$3(t-1) \binom{t}{2}$ defined parallel classes of Type 3 in the construction of $\dP$.
\end{proof}

\begin{lemma}
\label{lem:T22_t2}
The set of quadruples from Group 3 which are used in Type 1 when $t \equiv 2~(\text{mod}~4)$,
$$
\{ \{ (x,i),(y,i),(x,i+1),(y,i+1) \} ~:~ x,y \in \Z_t ,~ i=0,2 \}  ~,
$$
form $t-1$ parallel classes of Type 3.
\end{lemma}

\begin{proof}
The number of all the pairs $\{x,y\}$ with $x, y \in \Z_t$, $x \neq y$, is $\binom{t}{2}$.
Consider a one-factorization $\cF$ of~$K_t$ and a given one-factor
$\cF_j$, $0\leq j\leq t-2$, one can form one parallel class
$$
\{ \{ (x,i),(y,i),(x,i+1),(y,i+1) \} ~:~ \{x,y\} \in \cF_j,~ i=0,2 \}  ~,
$$
Hence, the $t-1$ one-factors of $\cF$ induce $t-1$ parallel classes of Type 3.
\end{proof}

\subsubsection{Type 3 for odd $t$}
\label{sec:T22odd}

For odd $t$ a near-one-factorization $\cF =\{ \cF_0,\cF_1,\ldots,\cF_{t-1} \}$ of $K_t$ is required.
Let $\cG_1$, $\cG_2$, $\cG_3$ be three pairwise disjoint sets, each one of size $\frac{t-1}{2} t^2$,
of quadruples from configuration (1,1,1,1) with the following properties. Each pair $\{ (x,0),(y,1) \}$
and each pair $\{ (x,2),(y,3) \}$, $x,y \in \Z_t$, is contained in exactly $\frac{t-1}{2}$ quadruples
of~$\cG_1$. Each pair $\{ (x,0),(y,3) \}$
and each pair $\{ (x,1),(y,2) \}$, $x,y \in \Z_t$, is contained in exactly $\frac{t-1}{2}$ quadruples
of~$\cG_2$. Each pair $\{ (x,0),(y,2) \}$
and each pair $\{ (x,1),(y,3) \}$, $x,y \in \Z_t$, is contained in exactly $\frac{t-1}{2}$ quadruples
of~$\cG_3$.

For a quadruple $B = \{ (i,0),(j,1),(k,2),(m,3) \}$ of $\cG_1$,
assume that $B$ contains the $r$-th appearance of the pair $\{ (i,0),(j,1) \}$ in $\cG_1$ and the $s$-th appearance
of the pair $\{ (k,2),(m,3) \}$ in $\cG_1$. Define the set of quadruples which contains $B$ and the quadruples
of the following sets.
$$
\{  \{ (x ,0),(y,0),(z,1),(v,1) \}  ~:~
\{ x,y \} = \cF_{i,p}, ~\{z,v\} = \cF_{j,p+r}, ~  0 \leq p \leq \frac{t-3}{2} \}~.
$$
$$
\{  \{ (x ,2),(y,2),(z,3),(v,3) \}  ~:~
\{ x,y \} = \cF_{k,p}, ~\{z,v\} = \cF_{m,p+s}, ~ 0 \leq p \leq \frac{t-3}{2} \}~.
$$

For a quadruple $B = \{ (i,0),(j,1),(k,2),(m,3) \}$ of $\cG_2$,
assume that $B$ contains the $r$-th appearance of the pair $\{ (i,0),(m,3) \}$ in $\cG_2$ and the $s$-th appearance
of the pair $\{ (j,1),(k,2) \}$ in $\cG_2$. Define the set of quadruples which contains $B$ and the quadruples
of the following sets.
$$
\{  \{ (x ,0),(y,0),(z,3),(v,3) \}  ~:~
\{ x,y \} = \cF_{i,p}, ~\{z,v\} = \cF_{m,p+r}, ~  0 \leq p \leq \frac{t-3}{2} \}~.
$$
$$
\{  \{ (x ,1),(y,1),(z,2),(v,2) \}  ~:~
\{ x,y \} = \cF_{j,p}, ~\{z,v\} = \cF_{k,p+s}, ~ 0 \leq p \leq \frac{t-3}{2} \}~.
$$

For a quadruple $B = \{ (i,0),(j,1),(k,2),(m,3) \}$ of $\cG_3$,
assume that $B$ contains the $r$-th appearance of the pair $\{ (i,0),(k,2) \}$ in $\cG_3$ and the $s$-th appearance
of the pair $\{ (j,1),(m,3) \}$ in $\cG_3$. Define the set of quadruples which contains $B$ and the quadruples
of the following sets.
$$
\{  \{ (x ,0),(y,0),(z,2),(v,2) \}  ~:~
\{ x,y \} = \cF_{i,p}, ~\{z,v\} = \cF_{k,p+r}, ~  0 \leq p \leq \frac{t-3}{2} \}~.
$$
$$
\{  \{ (x ,1),(y,1),(z,3),(v,3) \}  ~:~
\{ x,y \} = \cF_{j,p}, ~\{z,v\} = \cF_{m,p+s}, ~ 0 \leq p \leq \frac{t-3}{2} \}~.
$$

The following result can be readily verified.
\begin{lemma}
\label{lem:T22odd}
$~$
\begin{enumerate}
\item For each block $B = \{ (i,0),(j,1),(k,2),(m,3) \}$ in $\cG_1$, $\cG_2$ or $\cG_3$,
the $t$ defined quadruples form a parallel class with $t-1$ quadruples from Group 3 and one quadruple from Group 5.

\item If $t$ is odd, then each quadruple from the configurations of Group 3 appears either in one
of the defined parallel classes or in the parallel classes of Type 1.

\item The number of defined parallel classes of Type 3 for odd $t$ is $3 t \binom{t}{2}$.
\end{enumerate}
\end{lemma}

\begin{lemma}
\label{lem:T22_t1}
The set of quadruples from Group 3 contained in the parallel classes of Type 1 when $t \equiv 1(\text{mod}~4)$,
$$
\{ \{ (x,i),(y,i),(x,i+1),(y,i+1) \} ~:~ x,y \in \Z_t ,~ i \in \Z_4 \}  ~,
$$
$$
\{ \{ (x,i),(y,i),(x,i+2),(y,i+2) \} ~:~ x,y \in \Z_t ,~ i \in \{ 0,1 \} \}  ~,
$$
form $3t$ parallel classes of Type 3. Each one of these $3t$ parallel classes contains
exactly one quadruple from configuration $(1,1,1,1)$.
\end{lemma}
\begin{proof}
By Lemma \ref{lem:rm22a31a1111f4}, when $t \equiv 1(\text{mod}~4)$, the quadruples from Group 3 contained in the parallel classes of Type~1 are
$$
\{ \{ (x,i),(y,i),(x,i+1),(y,i+1) \} ~:~ x,y \in \Z_t ,~ i \in \Z_4 \}  ~,
$$
$$
\{ \{ (x,i),(y,i),(x,i+2),(y,i+2) \} ~:~ x,y \in \Z_t ,~ i \in \{ 0,1 \} \}  ~.
$$
$~$
Consider a near-one-factorization $\cF$ of $K_t$ and a given near-one-factor
$\cF_j$ of $\cF$, $0\leq j\leq t-1$, we can form three parallel classes
$$
\{ \{ (x,0),(y,0),(x,1),(y,1) \}, \{ (x,2),(y,2),(x,3),(y,3) \}  ~:~ \{x,y\} \in \cF_j,~ \}\cup \{\{ (j,0),(j,1),(j,2),(j,3)\}\}
$$
$$
\{ \{ (x,0),(y,0),(x,2),(y,2) \}, \{ (x,1),(y,1),(x,3),(y,3) \}  ~:~ \{x,y\} \in \cF_j,~ \}\cup \{\{ (j,0),(j,1),(j,2),(j,3)\}\}
$$
$$
\{ \{ (x,0),(y,0),(x,3),(y,3) \}, \{ (x,1),(y,1),(x,2),(y,2) \}  ~:~ \{x,y\} \in \cF_j,~ \}\cup \{\{ (j,0),(j,1),(j,2),(j,3)\}\}
$$
Hence, the related $t$ near-one-factors induce $3t$ parallel classes of Type~3 and
each one of these $3t$ parallel classes contains
exactly one quadruple from configuration $(1,1,1,1)$.
\end{proof}

\subsection{Parallel Classes of Type 4}
\label{sec:T211}
For Type 4,
the configurations of Group 4 are partitioned into six subsets, $\cM_1$, $\cM_2$, $\cM_3$, $\cM_4$, $\cM_5$, and $\cM_6$,
where $\cM_1 = \{ (2,1,1,0),(0,1,1,2) \}$, $\cM_2 = \{ (2,1,0,1),(0,1,2,1) \}$,
$\cM_3 = \{ (2,0,1,1),(0,2,1,1) \}$, $\cM_4 = \{ (1,2,0,1),(1,0,2,1) \}$, $\cM_5 = \{ (1,2,1,0),(1,0,1,2) \}$, and
$\cM_6 = \{ (1,1,2,0),(1,1,0,2) \}$. In the construction we distinguish between even $t$ and odd $t$.

\subsubsection{Type 4 for even $t$}
\label{sec:T211even}

Let $\cF = \{ \cF_0,\cF_1,\ldots,\cF_{t-2} \}$ be a one-factorization of $K_t$.
For $\cM_1$, let $\cU_1$ be a set with $(t-1)t^2$ quadruples defined by
$$
\cU_1 \triangleq \{ \{ (x,0),(y,1),(z,2),(x,3) \}  ~:~ x \in \Z_{t-1},~ y,z \in \Z_t \}.
$$
For each $x \in \Z_{t-1}$ and each pair $\{ (y,1),(z,2) \} \in \Z_t \times \{1,2\}$ there is exactly
one quadruple $\{ (x,0),(y,1),(z,2),(x,3) \} \in \cU_1$ which contains the triple $\{ (x,0),(y,1),(z,2) \}$ and exactly
one quadruple $\{ (x,0),(y,1),(z,2),(x,3) \} \in \cU_1$ which contains the triple $\{ (y,1),(z,2),(x,3) \}$.

Given a quadruple $X=\{ (i,0),(j,1),(k,2),(i,3) \}$ of the set $\cU_1$ we form
the following set $\cR_{i,j,k}$ in $\dP$, $i \in \Z_{t-1}$, $j,k \in \Z_t$, which contains $t$ quadruples
\begin{footnotesize}
$$
\{\{(x,0),(y,0),(j+r,1),(k+r,2)\},\{(x,3),(y,3),(j-r-1,1),(k-r-1,2)\}~~:~~ \{x,y\}=\cF_{i,r}, 0\leq r\leq \frac{t-2}{2}\}.
$$
\end{footnotesize}

For $\cM_i$, $2 \leq i \leq 6$ a set $\cU_i$ is defined similarly as follows.

$$
\cU_2 \triangleq \{ \{ (x,0),(y,1),(x,2),(z,3) \}  ~:~ x \in \Z_{t-1},~ y,z \in \Z_t \}.
$$
For each $x \in \Z_{t-1}$ and each pair $\{ (y,1),(z,3) \} \in \Z_t \times \{1,3\}$ there is exactly
one quadruple $\{ (x,0),(y,1),(x,2),(z,3) \} \in \cU_2$ which contains the triple $\{ (x,0),(y,1),(z,3) \}$ and exactly
one quadruple $\{ (x,0),(y,1),(x,2),(z,3) \} \in \cU_2$ which contains the triple $\{ (y,1),(x,2),(z,3) \}$.

$$
\cU_3 \triangleq \{ \{ (x,0),(x,1),(y,2),(z,3) \}  ~:~ x \in \Z_{t-1},~ y,z \in \Z_t \}.
$$
For each $x \in \Z_{t-1}$ and each pair $\{ (y,2),(z,3) \} \in \Z_t \times \{2,3\}$ there is exactly
one quadruple $\{ (x,0),(x,1),(y,2),(z,3) \} \in \cU_3$ which contains the triple $\{ (x,0),(y,2),(z,3) \}$ and exactly
one quadruple $\{ (x,0),(x,1),(y,2),(z,3) \} \in \cU_3$ which contains the triple $\{ (x,1),(y,2),(z,3) \}$.

$$
\cU_4 \triangleq \{ \{ (y,0),(x,1),(x,2),(z,3) \}  ~:~ x \in \Z_{t-1},~ y,z \in \Z_t \}.
$$
For each $x \in \Z_{t-1}$ and each pair $\{ (y,0),(z,3) \} \in \Z_t \times \{0,3\}$ there is exactly
one quadruple $\{ (y,0),(x,1),(x,2),(z,3) \} \in \cU_4$ which contains the triple $\{ (y,0),(x,1),(z,3) \}$ and exactly
one quadruple $\{ (y,0),(x,1),(x,2),(z,3) \} \in \cU_4$ which contains the triple $\{ (y,0),(x,2),(z,3) \}$.

$$
\cU_5 \triangleq \{ \{ (y,0),(x,1),(z,2),(x,3) \}  ~:~ x \in \Z_{t-1},~ y,z \in \Z_t \}.
$$
For each $x \in \Z_{t-1}$ and each pair $\{ (y,0),(z,2) \} \in \Z_t \times \{0,2\}$ there is exactly
one quadruple $\{ (y,0),(x,1),(z,2),(x,3) \} \in \cU_5$ which contains the triple $\{ (y,0),(x,1),(z,2) \}$ and exactly
one quadruple $\{ (y,0),(x,1),(z,2),(x,3) \} \in \cU_5$ which contains the triple $\{ (y,0),(z,2),(x,3) \}$.

$$
\cU_6 \triangleq \{ \{ (y,0),(z,1),(x,2),(x,3) \}  ~:~ x \in \Z_{t-1},~ y,z \in \Z_t \}.
$$
For each $x \in \Z_{t-1}$ and each pair $\{ (y,0),(z,1) \} \in \Z_t \times \{0,1\}$ there is exactly
one quadruple $\{ (y,0),(z,1),(x,2),(x,3) \} \in \cU_6$ which contains the triple $\{ (y,0),(z,1),(x,2) \}$ and exactly
one quadruple $\{ (y,0),(z,1),(x,2),(x,3) \} \in \cU_6$ which contains the triple $\{ (y,0),(z,1),(x,3) \}$.

Similarly to the definition of $\cR_{i,j,k}$ for a quadruple $\{ (i,0),(j,1),(k,2),(i,3) \} \in \cU_1$,
a related set of quadruple is defined for each element of $\cU_2$ (also $\cU_3$, $\cU_4$, $\cU_5$, and $\cU_6$),
where the definition is changed related to the positions of the 2's in the configurations of $\cM_2$
($\cM_3$, $\cM_4$, $\cM_5$, and $\cM_6$, respectively).

\begin{lemma}
\label{lem:T211even}
$~$
\begin{enumerate}
\item Each set $\cR_{i,j,k}$ defined by $\cU_1$ is a parallel class of Type 4.

\item Each quadruple from configurations $(2,1,1,0)$ and $(0,1,1,2)$ is contained in exactly one
of the defined parallel classes.

\item The total number of parallel classes of Type 4 (for $\bigcup_{i=1}^6 \cU_i$) is $6 (t-1) t^2$.
Each such parallel class contains only quadruples from Group 4. It contains $t$ quadruples from Group 4.
\end{enumerate}
\end{lemma}
\begin{proof}
Claim 1 and 2 can be easily verified.

For Claim 3, given any quadruples $X=\{(i,0),(j,1),(k,2),(i,3)\}$ of the set $\cU_1$ where $i\in \Z_{t-1}$ and $j,k\in \Z_t$,
one-factor of $\cF_i$ of $K_t$ is used to form a set $\cR_{i,j,k}$ in $\dP$.
It is easy to verify that $\cR_{i,j,k}$ is a parallel class of Type 4 in $\dP$. Hence, there are $(t-1) t^2$ parallel classes of Type 4
obtained from the set $\cU_1$. Similarly, for each $\cU_i$, $2\leq i\leq 6$, there are  $(t-1) t^2$ parallel classes of Type 4.
Thus, there are $6(t-1) t^2$ parallel classes of Type 4.
\end{proof}

\begin{example}
\label{ex:T211even}

For $t=4$ the set $\cU_1$ contains the following 48 quadruples
\begin{footnotesize}
$$
\{ (0,0),(0,1),(0,2),(0,3) \},\{ (0,0),(0,1),(1,2),(0,3) \},\{ (0,0),(0,1),(2,2),(0,3) \}\{ (0,0),(0,1),(3,2),(0,3) \},
$$
$$
\{ (0,0),(1,1),(0,2),(0,3) \},\{ (0,0),(1,1),(1,2),(0,3) \},\{ (0,0),(1,1),(2,2),(0,3) \}\{ (0,0),(1,1),(3,2),(0,3) \},
$$
$$
\{ (0,0),(2,1),(0,2),(0,3) \},\{ (0,0),(2,1),(1,2),(0,3) \},\{ (0,0),(2,1),(2,2),(0,3) \}\{ (0,0),(2,1),(3,2),(0,3) \},
$$
$$
\{ (0,0),(3,1),(0,2),(0,3) \},\{ (0,0),(3,1),(1,2),(0,3) \},\{ (0,0),(3,1),(2,2),(0,3) \}\{ (0,0),(3,1),(3,2),(0,3) \},
$$
$$
\{ (1,0),(0,1),(0,2),(1,3) \},\{ (1,0),(0,1),(1,2),(1,3) \},\{ (1,0),(0,1),(2,2),(1,3) \}\{ (1,0),(0,1),(3,2),(1,3) \},
$$
$$
\{ (1,0),(1,1),(0,2),(1,3) \},\{ (1,0),(1,1),(1,2),(1,3) \},\{ (1,0),(1,1),(2,2),(1,3) \}\{ (1,0),(1,1),(3,2),(1,3) \},
$$
$$
\{ (1,0),(2,1),(0,2),(1,3) \},\{ (1,0),(2,1),(1,2),(1,3) \},\{ (1,0),(2,1),(2,2),(1,3) \}\{ (1,0),(2,1),(3,2),(1,3) \},
$$
$$
\{ (1,0),(3,1),(0,2),(1,3) \},\{ (1,0),(3,1),(1,2),(1,3) \},\{ (1,0),(3,1),(2,2),(1,3) \}\{ (1,0),(3,1),(3,2),(1,3) \},
$$
$$
\{ (2,0),(0,1),(0,2),(2,3) \},\{ (2,0),(0,1),(1,2),(2,3) \},\{ (2,0),(0,1),(2,2),(2,3) \}\{ (2,0),(0,1),(3,2),(2,3) \},
$$
$$
\{ (2,0),(1,1),(0,2),(2,3) \},\{ (2,0),(1,1),(1,2),(2,3) \},\{ (2,0),(1,1),(2,2),(2,3) \}\{ (2,0),(1,1),(3,2),(2,3) \},
$$
$$
\{ (2,0),(2,1),(0,2),(2,3) \},\{ (2,0),(2,1),(1,2),(2,3) \},\{ (2,0),(2,1),(2,2),(2,3) \}\{ (2,0),(2,1),(3,2),(2,3) \},
$$
$$
\{ (2,0),(3,1),(0,2),(2,3) \},\{ (2,0),(3,1),(1,2),(2,3) \},\{ (2,0),(3,1),(2,2),(2,3) \}\{ (2,0),(3,1),(3,2),(2,3) \}.
$$
\end{footnotesize}

The first 8 quadruples of $\cU_1$ with the one-factorization $\cF=\{\cF_0,\cF_1,\cF_2\}$,
where $\cF_0 = \{ \cF_{0,0} =\{1,2\},\cF_{0,1}=\{0,3\} \}$, $\cF_1 = \{ \cF_{1,0} =\{0,2\},\cF_{1,1}=\{1,3\} \}$,
$\cF_2 = \{ \cF_{2,0} =\{0,1\},\cF_{2,1}=\{2,3\} \}$, yield the following 8 parallel classes.

From the quadruple $\{ (0,0),(0,1),(0,2),(0,3) \}$ we form the parallel class
\begin{footnotesize}
$$
\{ \{(1,0),(2,0),(0,1),(0,2)\},\{(0,0),(3,0),(1,1),(1,2)\},\{(1,3),(2,3),(3,1),(3,2)\},\{(0,3),(3,3),(2,1),(2,2)\}  \}.
$$
\end{footnotesize}
From the quadruple $\{ (0,0),(0,1),(1,2),(0,3) \}$ we form the parallel class
\begin{footnotesize}
$$
\{ \{(1,0),(2,0),(0,1),(1,2)\},\{(0,0),(3,0),(1,1),(2,2)\},\{(1,3),(2,3),(3,1),(0,2)\},\{(0,3),(3,3),(2,1),(3,2)\}  \}.
$$
\end{footnotesize}
From the quadruple $\{ (0,0),(0,1),(2,2),(0,3) \}$ we form the parallel class
\begin{footnotesize}
$$
\{ \{(1,0),(2,0),(0,1),(2,2)\},\{(0,0),(3,0),(1,1),(3,2)\},\{(1,3),(2,3),(3,1),(1,2)\},\{(0,3),(3,3),(2,1),(0,2)\}  \}.
$$
\end{footnotesize}
From the quadruple $\{ (0,0),(0,1),(3,2),(0,3) \}$ we form the parallel class
\begin{footnotesize}
$$
\{ \{(1,0),(2,0),(0,1),(3,2)\},\{(0,0),(3,0),(1,1),(0,2)\},\{(1,3),(2,3),(3,1),(2,2)\},\{(0,3),(3,3),(2,1),(1,2)\}  \}.
$$
\end{footnotesize}
From the quadruple $\{ (0,0),(1,1),(0,2),(0,3) \}$ we form the parallel class
\begin{footnotesize}
$$
\{ \{(1,0),(2,0),(1,1),(0,2)\},\{(0,0),(3,0),(2,1),(1,2)\},\{(1,3),(2,3),(0,1),(3,2)\},\{(0,3),(3,3),(3,1),(2,2)\}  \}.
$$
\end{footnotesize}
From the quadruple $\{ (0,0),(1,1),(1,2),(0,3) \}$ we form the parallel class
\begin{footnotesize}
$$
\{ \{(1,0),(2,0),(1,1),(1,2)\},\{(0,0),(3,0),(2,1),(2,2)\},\{(1,3),(2,3),(0,1),(0,2)\},\{(0,3),(3,3),(3,1),(3,2)\}  \}.
$$
\end{footnotesize}
From the quadruple $\{ (0,0),(1,1),(2,2),(0,3) \}$ we form the parallel class
\begin{footnotesize}
$$
\{ \{(1,0),(2,0),(1,1),(2,2)\},\{(0,0),(3,0),(2,1),(3,2)\},\{(1,3),(2,3),(0,1),(1,2)\},\{(0,3),(3,3),(3,1),(0,2)\}  \}.
$$
\end{footnotesize}
From the quadruple $\{ (0,0),(1,1),(3,2),(0,3) \}$ we form the parallel class
\begin{footnotesize}
$$
\{ \{(1,0),(2,0),(1,1),(3,2)\},\{(0,0),(3,0),(2,1),(0,2)\},\{(1,3),(2,3),(0,1),(2,2)\},\{(0,3),(3,3),(3,1),(1,2)\}  \}.
$$
\end{footnotesize}

\end{example}

\subsubsection{Type 4 for odd $t$}
\label{sec:T211odd}

Let $\cF= \{ \cF_0,\cF_1,\ldots,\cF_{t-1}\}$ be a near-one-factorization of $K_t$ and
let $\{ \cH_i ~:~ 1 \leq i \leq 6\}$ be a set with six pairwise disjoint subsets from Group 5, i.e. configuration (1,1,1,1),
where the subset $\cH_i$ is associated with the subset $\cM_i$, for $1 \leq i \leq 6$.
For example $\cH_1$ is a set with $t^3$ quadruples defined by
$$
\cH_1 \triangleq \{ \{ (i,0),(j,1),(k,2),(i+j+k,3) \}  ~:~ i,j,k \in \Z_t \}.
$$
For each $i,j,k \in \Z_t$ there is exactly
one quadruple $\{ (i,0),(j,1),(k,2),(\ell,3) \} \in \cH_1$ which contains the triple $\{ (i,0),(j,1),(k,2) \}$ and exactly
one quadruple $\{ (i,0),(j,1),(j,2),(\ell,3) \} \in \cH_1$ which contains the triple $\{ (j,1),(k,2),(\ell,3) \}$.

Related subsets $\cH_i$'s  of quadruples are
defined for each $\cM_i$, for each $2 \leq i \leq 6$.

\begin{remark}
The $\cH_i$'s have similar properties as the $\cU_i$'s. But, the $\cH_i$'s are pairwise disjoint,
while the $\cU_i$'s might not be disjoint as they were defined.
\end{remark}

Given a quadruple $X=\{ (i,0),(j,1),(k,2),(\ell,3) \}$ of the subset $\cH_1$ we form
a set $\cR_{i,j,k}$ in $\dP$ which contains $X$ and the following set with $t-1$ quadruples
$$
\{ \{ (x,0),(y,0),(j +r+1,1),(k+r+1,2) \} ~:~ \{ x,y \} = \cF_{i,r}, ~ 0 \leq r \leq \frac{t-3}{2} \}
$$
$$
\bigcup \{ \{ (j-r-1,1),(k-r-1,2),(x,3),(y,3)\} ~:~ \{ x,y \} = \cF_{\ell,r}, ~ 0 \leq r \leq \frac{t-3}{2} \}.
$$
Related subsets of quadruples are defined for each $\cH_i$ and $\cM_i$, for each $2 \leq i \leq 6$.

The following lemma has a proof similar to the one of Lemma~\ref{lem:T211even}.
\begin{lemma}
\label{lem:T211odd}
$~$
\begin{enumerate}
\item Each set $\cR_{i,j,k}$ defined by $\cH_1$ is a parallel class of Type 4.

\item Each quadruple from configurations $(2,1,1,0)$ and $(0,1,1,2)$ is contained in exactly one
of the defined parallel classes.

\item The total number of parallel classes of Type 4 (for $\bigcup_{i=1}^6 \cH_i$) is $6 t^3$.
Each such parallel class contains $t-1$ quadruples from Group 4 and one quadruple from Group 5.
\end{enumerate}
\end{lemma}

%

\subsection{Parallel Classes of Type 5}
\label{sec:1111}

Type 5 contains quadruples only from Group 5 which has the unique configuration $(1,1,1,1)$.
There are $t^4$ quadruples in this configuration. Each parallel class contains $t$ quadruples and
hence there are $t^3$ possible pairwise disjoint parallel classes which contain all the quadruples from Group~5.
They can be easily constructed in a few ways.
But, some of them might have been used with quadruples from other configurations
in Type 1, Type 3, or Type 4. Hence, our construction will take
into account the quadruples which are used with the configurations in other types. We start with
the following four subsets of quadruples from configuration $(1,1,1,1)$.
$$
\cA \triangleq \{ \{ (i,0),(i,1),(i,2),(i,3) \} ~:~ i \in \Z_t \},
$$
$$
\cB \triangleq \{ \{ (i,0),(0,1),(i,2),(0,3) \} ~:~ i \in \Z_t \},
$$
$$
\cC \triangleq \{ \{ (i,0),(i,1),(0,2),(0,3) \} ~:~ i \in \Z_t \},
$$
$$
\cD \triangleq \{ \{ (i,0),(0,1),(0,2),(0,3) \} ~:~ i \in \Z_t \}.
$$

For $X= \{ (x_0,0),(x_1,1),(x_2,2),(x_3,3) \}$ and $Y= \{ (y_0,0),(y_1,1),(y_2,2),(y_3,3) \}$
the addition $X+Y$ is defined by
$$
X+Y \triangleq \{ (x_0+y_0,0),(x_1+y_1,1),(x_2+y_2,2),(x_3+y_3,3)  \}.
$$
Recall that the computation in the first coordinate is performed modulo $t$.

For two subsets of quadruples $\cS_1$ and $\cS_2$ from configuration $(1,1,1,1)$, the addition $\cS_1 + \cS_2$ is defined by
$$
\cS_1 + \cS_2 \triangleq \{ X+Y ~:~ X \in \cS_1, ~ Y \in \cS_2 \}.
$$

\begin{lemma}
\label{lem:allt4}
The set of quadruples
$$
\cA + \cB + \cC + \cD = \{ X+Y+Z+V ~:~ X\in \cA, ~Y \in \cB,~ Z\in \cC,~ V \in \cD \}
$$
consists  of all the $t^4$ quadruples of $\Z_t \times Z_4$ from configuration $(1,1,1,1)$.
\end{lemma}
\begin{proof}
First, assume for the contrary that two quadruples of $\Z_t \times \Z_4$ are generated in two
different ways in $\cA + \cB + \cC + \cD$. Assume that there exist eight parameters $i_1,i_2,i_3,i_4$,
and $j_1,j_2,j_3,j_4$ such that $(i_1,i_2,i_3,i_4) \neq (j_1,j_2,j_3,j_4)$ and the two quadruples
\begin{footnotesize}
$$
\{ (i_1,0),(i_1,1),(i_1,2),(i_1,3) \} + \{ (i_2,0),(0,1),(i_2,2),(0,3) \} + \{ (i_3,0),(i_3,1),(0,2),(0,3) \} + \{ (i_4,0),(0,1),(0,2),(0,3) \}
$$
\end{footnotesize}
and
\begin{footnotesize}
$$
\{ (j_1,0),(j_1,1),(j_1,2),(j_1,3) \} + \{ (j_2,0),(0,1),(j_2,2),(0,3) \} + \{ (j_3,0),(j_3,1),(0,2),(0,3) \} + \{ (j_4,0),(0,1),(0,2),(0,3) \},
$$
\end{footnotesize}
are equal. This implies that
$$
i_1 + i_2 + i_3 + i_4 \equiv j_1 + j_2 + j_3 + j_4 ~ (\text{mod}~t),
$$
$$
i_1 + i_3 \equiv j_1 +j_3 ~ (\text{mod}~t),
$$
$$
i_1 + i_2 \equiv j_1 + j_2 ~ (\text{mod}~t),
$$
$$
i_1 \equiv j_1 ~ (\text{mod}~t).
$$
This set of equations has exactly one solution which is $(i_1,i_2,i_3,i_4) = (j_1,j_2,j_3,j_4)$, a contradiction.
Hence, all the $t^4$ quadruples formed by $X+Y+Z+V$, where $X\in \cA$, $Y \in \cB$, $Z\in \cC$, and $V \in \cD$ are distinct,
which implies the claim of the lemma.
\end{proof}

The following lemma can be easily verified.
\begin{lemma}
\label{lem:parallel1111}
The set of quadruples in $\cA$ forms a parallel class and for any quadruple $X$ from configuration $(1,1,1,1)$ the set
$$
X + \cA \triangleq \{ X + Y ~:~ Y \in \cA \}
$$
of $t$ quadruples forms a parallel class from configuration $(1,1,1,1)$.
\end{lemma}

We are now in position to define the parallel classes of Type 5, i.e. the parallel classes
with quadruples only from configuration $(1,1,1,1)$. For even $t$, this is rather simple since there are no
quadruples from configuration $(1,1,1,1)$ in Types 1, 2, 3, and 4. Since by Lemma~\ref{lem:allt4}, $\cA+\cB+\cC+\cD$ contains
all the quadruples from configuration $(1,1,1,1)$, it follows from Lemma~\ref{lem:parallel1111} that for each quadruples
$X \in \cB+\cC+\cD$, the set $X + \cA$ is a parallel class. This implies that

\begin{lemma}
\label{lem:T1111even}
For even $t$, the set of quadruples from configuration $(1,1,1,1)$ can be partitioned into $t^3$
parallel classes, each one of size $t$.
\end{lemma}

\subsubsection{Type 5 for odd $t$}
\label{sec:1111oddt}

The situation is more complicated in the definition of the quadruples of Type 5 for odd $t$.
We partition the $t^4$ quadruples from configuration $(1,1,1,1)$ into five subsets,
$\cL_1$, $\cL_2$, $\cL_3$, $\cL_4$, $\cL_5$, as follows.
$$
\cL_1 \triangleq \{ \{ (i,0),(0,1),(0,2),(0,3)\} ~:~ 1 \leq i \leq 6 \} + \cA + \cB + \cC
$$
and it is partitioned into six subsets $\cH_1$, $\cH_2$, $\cH_3$, $\cH_4$, $\cH_5$, $\cH_6$, where
$$
\cH_i \triangleq \{ (i,0),(0,1),(0,2),(0,3)\}  + \cA + \cB + \cC,~~~ 1 \leq i \leq 6.
$$

$$
\cL_2 \triangleq \{ \{ (i,0),(i,1),(0,2),(0,3)\} ~:~ 1 \leq i \leq \frac{t-1}{2} \} + \cA + \cB,
$$
$$
\cL_3 \triangleq \{ \{ (i,0),(i,1),(0,2),(0,3)\} ~:~ \frac{t+1}{2} \leq i \leq t-1 \} + \cA + \cB,
$$

$$
\cL_4 \triangleq \{ \{ (i,0),(0,1),(0,2),(0,3)\} ~:~ 7 \leq i \leq \frac{t+11}{2} \} + \cA + \cC,
$$

$$
\cL_5 \triangleq (\cA + \cB + \cC + \cD) \setminus \bigcup_{i=1}^4 \cL_i .
$$

\begin{lemma}
\label{lem:allG5}
$~$
\begin{enumerate}
\item The subsets $\cL_1$, $\cL_2$, $\cL_3$, $\cL_4$, and $\cL_5$ are pairwise disjoint.
\item The subset $\cL_5$ can be partitioned into parallel classes.
\end{enumerate}
\end{lemma}
\begin{proof}
By Lemma~\ref{lem:allt4} we have that
\begin{equation}
\label{eq:eq1}
\Z_t \times \Z_4 = \cA + \cB + \cC + \cD= \{ \{ (i_1,0),(i_2,1),(i_3,2),(i_4,3)\} ~:~ 0 \leq i_1,i_2,i_3,i_4 \leq t-1 \}.
\end{equation}
It is easy to verify that each one of quadruples in $\Z_t \times \Z_4$ is contained
in exactly one of the four subsets of quadruples
\begin{equation}
\label{eq:eq2}
\cA + \cB + \cC, ~~ \cL_1, ~~ \cA + \cB + \cC + \cD_1, ~~ \cA + \cB + \cC + \cD_2,
\end{equation}
where $\cD_1 \triangleq \{ \{ (i,0),(0,1),(0,2),(0,3)\} ~:~ 7 \leq i \leq \frac{t+11}{2} \}$,
$\cD_2 \triangleq \{ \{ (i,0),(0,1),(0,2),(0,3)\} ~:~ \frac{t+13}{2} \leq i \leq t-1 \}$,
and as defined before $\cL_1 = \{ \{ (i,0),(0,1),(0,2),(0,3)\} ~:~ 1 \leq i \leq 6 \} +\cA + \cB + \cC$.

Similarly by the definition of $\cL_2$ and $\cL_3$ we have that
\begin{equation}
\label{eq:eq3}
\cA + \cB + \cC = (\cA + \cB) \cup \cL_2 \cup \cL_3 ~.
\end{equation}
Similar arguments and the definitions of $\cD_1$ and $\cL_4$ imply that
\begin{equation}
\label{eq:eq4}
\cA + \cB + \cC + \cD_1 = \cL_4 \cup (\cA + \{ \{ (i,0),(0,1),(i,2),(0,3) \} ~:~ 1 \leq i \leq t-1  \} + \cC + \cD_1).
\end{equation}
Therefore, by Equations (\ref{eq:eq1}), (\ref{eq:eq2}), (\ref{eq:eq3}), and (\ref{eq:eq4}), we have
$$
\cL_5 = \Z_t \times \Z_4 \setminus (\cL_1 \cup \cL_2 \cup \cL_3 \cup \cL_4)
$$
$$
= (\cA + \cB) \cup (\cA + \{ \{ (i,0),(0,1),(i,2),(0,3) \} ~:~ 1 \leq i \leq t-1  \} + \cC + \cD_1) \cup (\cA + \cB + \cC + \cD_2).
$$
Hence $\cL_5 = \cS + \cA$ for some subset $\cS$ of quadruples from configuration $(1,1,1,1)$ which implies
by Lemma~\ref{lem:parallel1111} that $\cL_5$ can be partitioned into pairwise disjoint parallel classes.
\end{proof}

\begin{lemma}
For any integer $r$, the subset $\cS \triangleq \{ (r,0),(0,1),(0,2),(0,3)\}  + \cA + \cB + \cC$ of configuration $(1,1,1,1)$
contains $t^3$ quadruples and each triple $\{ (x_1,i),(x_2,j),(x_3,k)\}$,
where $|\{i,j,k\}|=3$ is contained in exactly one quadruple of $\cS$.
\end{lemma}
\begin{proof}
Consider the matrix
$$
G=\left[
\begin{array}{cccc}
1 & 1 & 1 & 1 \\
1 & 0 & 1 & 0 \\
1 & 1 & 0 & 0
\end{array}
\right],
$$
and the set of codewords in the code $\hat{C} \triangleq \{ (i_1,i_2,i_3) G ~:~ i_1,i_2,i_3 \in \Z_t \}$.
Note that $\cA + \cB + \cC = \{\{ (j_0,0),(j_1,1),(j_2,2),(j_3,3) \} ~:~ (j_0,j_1,j_2,j_3) \in \hat{C} \}$.
The claim follows now  from the observation that
the matrix defined by any three columns of $G$ is nonsingular.
\end{proof}

\begin{corollary}
\label{cor:Hsat}
The subsets $\cH_1,\cH_2,\cH_3,\cH_4,\cH_5,\cH_6$ satisfy the properties required for Type 4 when $t$ is odd.
These subsets of quadruples from Group 5 contained in the parallel classes of Type 4 form $6 t^2$ parallel classes
defined for Type 5 (configuration $(1,1,1,1)$).
\end{corollary}

\begin{lemma}
\label{lem:like_OA}
$~$
\begin{enumerate}
\item Each pair $\{ (y,0),(z,1) \}$, $y,z \in \Z_t$, is contained exactly $\frac{t-1}{2}$ times in the quadruples of $\cL_2$.
Each pair $\{ (y,2),(z,3) \}$, $y,z \in \Z_t$, is contained exactly $\frac{t-1}{2}$ times in the quadruples of $\cL_2$.

\item Each pair $\{ (y,0),(z,3) \}$, $y,z \in \Z_t$, is contained exactly $\frac{t-1}{2}$ times in the quadruples of $\cL_3$.
Each pair $\{ (y,1),(z,2) \}$, $y,z \in \Z_t$, is contained exactly $\frac{t-1}{2}$ times in the quadruples of $\cL_3$.

\item Each pair $\{ (y,0),(z,2) \}$, $y,z \in \Z_t$, is contained exactly $\frac{t-1}{2}$ times in the quadruples of $\cL_4$.
Each pair $\{ (y,1),(z,3) \}$, $y,z \in \Z_t$, is contained exactly $\frac{t-1}{2}$ times in the quadruples of $\cL_4$.
\end{enumerate}
\end{lemma}
\begin{proof}
First note that in $\cA$ there are $t$ quadruples and for any given $x \in \Z_t$ and $i \in \Z_4$, the element $(x,i)$
is contained in exactly one of the quadruples of $\cA$. In $\cB$ each pair $\{(y,0), (0,1)\}$ is contained in exactly one of the
quadruples. Hence, each pair $\{ (y,0),(z,1) \}$, $y,z \in \Z_t$, is contained in exactly one quadruple of $\cA + \cB$.
Therefore, each pair $\{ (y,0),(z,1) \}$, $y,z \in \Z_t$, is also contained in exactly
one quadruple of $\{ (j,0), (j,1), (0,2),(0,3) \} + \cA + \cB$, for any $j \in \Z_t$. Thus,
each pair $\{ (y,0),(z,1) \}$, $y,z \in \Z_t$, is contained in exactly $\frac{t-1}{2}$ times in the quadruples of $\cL_2$.

The other claims in the lemma are proved in exactly the same way.
\end{proof}
\begin{corollary}
\label{cor:Gsat}
The subset $\cL_2$ has the properties required for $\cG_1$ in Type 3 for odd $t$,
the subset $\cL_3$ has the properties required for $\cG_2$ in Type 3 for odd $t$,
and the subset $\cL_4$ has the properties required for $\cG_3$ in Type 3 for odd $t$.
These subsets of quadruples from Group 5 contained in the parallel classes of Type 3
form $3 \binom{t}{2}$ parallel classes of configuration $(1,1,1,1)$.
\end{corollary}

\begin{remark}
The definitions of $\cD_1$ and $\cD_2$ given in Lemma~\ref{lem:allG5} imply that the construction for odd $t$ works only if
$(t+13)/2 \leq t-1$ which implies that $t \geq 15$.
\end{remark}

\section{Proof of the Main Result}
\label{sec:main}

In this section we will prove Theorem~\ref{thm:main}, i.e.,
the existence of a quadrupling construction whenever $t \equiv 0,3,4,6,8,9~(\text{mod}~12)$.
For each residue modulo 12 a separate proof will be given.

\noindent
{\bf Case 1: $t \equiv 0~(\text{mod}~12)$.}

In this case the parallel classes are taken exactly as defined in Section~\ref{sec:const_par}.
By Lemma~\ref{lem:T4t0} there are $\binom{t-1}{3}$ parallel classes of Type 1.
By Lemma~\ref{lem:T31} there are $2t(t-1)(t-2)$ parallel classes of Type~2.
By Lemma~\ref{lem:T22even} there are $3(t-1) \binom{t}{2}$ parallel classes of Type 3.
By Lemma~\ref{lem:T211even} there are $6 (t-1) t^2$ parallel classes of Type 4.
By Lemma~\ref{lem:T1111even} there are $t^3$ parallel classes of Type 5.
All these parallel classes together imply that there are a total of
$$
\binom{t-1}{3} + 2t(t-1)(t-2) + 3(t-1) \binom{t}{2} + 6 (t-1) t^2 + t^3 = \binom{4t-1}{3}
$$
parallel classes as required for $\mathrm{BP}(4t,4)$. Thus, if there exist a $\mathrm{BP}(t,4)$ and a $\mathrm{BP}(t,3)$, then there exists a~$\mathrm{BP}(4t,4)$.

\noindent
{\bf Case 2: $t \equiv 3~(\text{mod}~12)$.}

In this case the parallel classes are taken exactly as defined in Section~\ref{sec:const_par},
except for the parallel classes of Type 5.
By Lemma~\ref{lem:T4t3} there are $\binom{t}{3}$ parallel classes of Type 1.
By Lemma~\ref{lem:T31} there are $(4t-1)\binom{t-1}{2}$ parallel classes of Type 2.
Lemma~\ref{lem:rm31f4}, Lemma~\ref{lem:T4t3}, and Lemma~\ref{lem:T31}, imply
that each quadruple from Group 1 and Group 2 is contained in exactly one of the
parallel classes of Type~1 or Type~2.
By Lemma~\ref{lem:T22odd} there are $3 t \binom{t}{2}$ parallel classes of Type 3.
In these $3 t \binom{t}{2}$ parallel classes there are $3 t \binom{t}{2}$ quadruples from Group 5.
By Lemma~\ref{lem:T211odd} there are $6 t^3$ parallel classes of Type 4.
In these $6 t^3$ parallel classes there are $6 t^3$ quadruples from Group 5.
By Lemma~\ref{lem:T22odd}, Lemma~\ref{lem:T211odd}, Lemma~\ref{lem:allG5}, Corollary~\ref{cor:Hsat}, and Corollary~\ref{cor:Gsat},
there are $t^3 - 3 \binom{t}{2} - 6 t^2$ parallel classes of Type 5.
All these parallel classes together imply that there are a total of
$$
\binom{t}{3} + (4t-1)\binom{t-1}{2} + 3 t \binom{t}{2} + 6  t^3 + t^3 - 3 \binom{t}{2} - 6 t^2 =  \binom{4t-1}{3}
$$
parallel classes as required for $\mathrm{BP}(4t,4)$. Thus, if there exist a $\mathrm{BP}(t+1,4)$ and a $\mathrm{BP}(t,3)$, then there exists a~$\mathrm{BP}(4t,4)$.

\noindent
{\bf Case 3: $t \equiv 4~(\text{mod}~12)$.}

By Theorem~\ref{thm:4_8m12}, if there exists a $\mathrm{BP}(t,4)$ then there exists a $\mathrm{BP}(2t,4)$, and since
$2t \equiv 8~(\text{mod}~12)$, Theorem~\ref{thm:4_8m12} also implies that there exists a $\mathrm{BP}(4t,4)$.

\noindent
{\bf Case 4: $t \equiv 6~(\text{mod}~12)$.}

In this case the parallel classes are taken exactly as defined in Section~\ref{sec:const_par},
except for the parallel classes of Type 3.
By Lemma~\ref{lem:T4t2} there are $\binom{t+1}{3}$ parallel classes of Type 1.
By Lemma~\ref{lem:T31} there are $(2t-1)(t-1)(t-2)$ parallel classes of Type 2.
Lemma~\ref{lem:rm22a31f4}, Lemma~\ref{lem:T4t2}, and Lemma~\ref{lem:T31}, imply
that each quadruple from Group 1 and Group 2 is contained in exactly one of the
parallel classes of Type~1 or Type~2.
By Lemma~\ref{lem:T22even} there are $3(t-1) \binom{t}{2}$ parallel classes of Type 3,
from which by Lemma~\ref{lem:T22_t2}, the quadruples of $t-1$ parallel classes are already contained in the parallel classes of Type 1.
By Lemma~\ref{lem:T211even} there are $6 (t-1) t^2$ parallel classes of Type 4.
By Lemma~\ref{lem:T1111even} there are $t^3$ parallel classes of Type 5.
All these parallel classes together imply that there are a total of
$$
\binom{t+1}{3} + (2t-1)(t-1)(t-2) + 3(t-1) \binom{t}{2} -(t-1) + 6 (t-1) t^2 + t^3 = \binom{4t-1}{3}
$$
parallel classes as required for $\mathrm{BP}(4t,4)$. Thus, if there exist a $\mathrm{BP}(t+2,4)$ and a $\mathrm{BP}(t,3)$, then there exists a~$\mathrm{BP}(4t,4)$.

\noindent
{\bf Case 5: $t \equiv 8~(\text{mod}~12)$.}

By Theorem~\ref{thm:4_8m12}, if there exists a $\mathrm{BP}(t,4)$ then there exists a $\mathrm{BP}(2t,4)$, and since
$2t \equiv 4~(\text{mod}~12)$, Theorem~\ref{thm:4_8m12} also implies that there exists a $\mathrm{BP}(4t,4)$.

\noindent
{\bf Case 6: $t \equiv 9~(\text{mod}~12)$.}

This case is the most complicated one. For this case we will have to change the definitions of the
parallel classes of Type 5 for odd $t$. More precisely, we will have to change the definition
of the partition of the $t^4$ quadruples from configuration $(1,1,1,1)$ into five subsets.
We define the following five subsets, where $\cA$, $\cB$, $\cC$, and $\cD$ are defined as
in Section~\ref{sec:1111}.

\begin{remark}
The following definition can also be used for other values of odd $t$. But, the previous definition
is much simpler to verify and can also be seen as an introduction for the more complicated defintion
considered in the next few paragraphs.
\end{remark}

Partition the $t^4$ quadruples from configuration $(1,1,1,1)$ into five subsets,
$\cL'_1$, $\cL'_2$, $\cL'_3$, $\cL'_4$, $\cL'_5$, as follows.
$$
\cL'_1 \triangleq \{ \{ (i,0),(0,1),(0,2),(0,3)\} ~:~ 1 \leq i \leq 6 \} + \cA + \cB + \cC
$$
and it is partitioned into six subsets $\cH_1$, $\cH_2$, $\cH_3$, $\cH_4$, $\cH_5$, $\cH_6$, where
$$
\cH_i \triangleq \{ (i,0),(0,1),(0,2),(0,3)\}  + \cA + \cB + \cC,~~~ 1 \leq i \leq 6.
$$

$$
\cL'_2 \triangleq \{ \{ (i,0),(i,1),(0,2),(0,3)\} ~:~ 1 \leq i \leq \frac{t-1}{2} \} + \cA + \cB,
$$

\begin{remark}
These definitions are the same one as were defined in Section~\ref{sec:1111oddt}.
\end{remark}

$$
\cL'_3 \triangleq \{ (-1,0),(0,1),(0,2),(0,3)\} + \{ \{ (i,0),(i,1),(0,2),(0,3)\} ~:~ \frac{t+1}{2} \leq i \leq t-1 \} + \cA + \cB,
$$

$$
\cL'_4 \triangleq ( \{ \{ (i,0),(0,1),(i,2),(0,3)\} ~:~  1 \leq i \leq \frac{t-1}{2} \} + \cA )
$$
$$
\cup  \{ (7,0),(0,1),(0,2),(0,3) \} + \{ \{ (i,0),(i,1),(0,2),(0,3)\} ~:~ 1 \leq i \leq t-1 ~ \text{and} ~ i \neq -7 \}
$$
$$
+ \{ \{ (i,0),(0,1),(i,2),(0,3)\} ~:~  1 \leq i \leq \frac{t-1}{2} \} + \cA
$$
$$
\cup  \{ (7,0),(-7,1),(0,2),(0,3) \} + \{ \{ (i,0),(0,1),(i,2),(0,3)\} ~:~  1 \leq i \leq \frac{t-1}{2} \} + \cA
$$

$$
\cL'_5 \triangleq (\cA + \cB + \cC + \cD) \setminus \bigcup_{i=1}^4 \cL'_i .
$$

\begin{lemma}
\label{lem:allG5B}
$~$
\begin{enumerate}
\item The subsets $\cL'_1$, $\cL'_2$, $\cL'_3$, $\cL'_4$, and $\cL'_5$ are pairwise disjoint.
\item The subset $\cL'_5$ can be partitioned into parallel classes.
\end{enumerate}
\end{lemma}
\begin{proof}
By Lemma~\ref{lem:allt4} we have that
\begin{equation}
\label{eq:eq1B}
\Z_t \times \Z_4 = \cA + \cB + \cC + \cD= \{ \{ (i_1,0),(i_2,1),(i_3,2),(i_4,3)\} ~:~ 0 \leq i_1,i_2,i_3,i_4 \leq t-1 \}.
\end{equation}
First, we have to show that the four subsets $\cL'_1$,
$\cL'_2$, $\cL'_3$, and $\cL'_4$ are pairwise disjoint. To see that note that
$$
\cL'_1 = \{ \{ (i,0),(0,1),(0,2),(0,3)\} ~:~ 1 \leq i \leq 6 \} + \cA + \cB + \cC,
$$
$$
\cL'_2 = \{ (0,0),(0,1),(0,2),(0,3)\} + \cA + \cB + \{ \{ (i,0),(i,1),(0,2),(0,3)\} ~:~ 1 \leq i \leq \frac{t-1}{2} \},
$$
$$
\cL'_3 = \{ (-1,0),(0,1),(0,2),(0,3)\} + \cA + \cB + \{ \{ (i,0),(i,1),(0,2),(0,3)\} ~:~ \frac{t+1}{2} \leq i \leq t-1 \},
$$
$$
\cL'_4 = \cD_1 \cup \cD_2 \cup \cD_3, ~ \text{where}
$$
$$
\cD_1 \subseteq \{ (0,0),(0,1),(0,2),(0,3)\} + \cA + \cB + \{ (0,0),(0,1),(0,2),(0,3)\},
$$
$$
\cD_2 \subseteq \{ (7,0),(0,1),(0,2),(0,3)\} + \cA + \cB + \{ \{ (i,0),(i,1),(0,2),(0,3)\} ~:~ 1 \leq i \leq t-1 ~ \text{and} ~ i \neq -7 \},
$$
$$
\cD_3 \subseteq \{ (14,0),(0,1),(0,2),(0,3)\} + \cA + \cB + \{ (-7,0),(-7,1),(0,2),(0,3)\}.
$$
With this representation it is readily verified that these four subsets are pairwise disjoint.
It follows that the five subsets $\cL'_1$, $\cL'_2$, $\cL'_3$, $\cL'_4$, and $\cL'_5$ are pairwise disjoint.

Now, it is readily verified that each subset $\cL'_1$, $\cL'_2$, $\cL'_3$, $\cL'_4$, and $\cL'_5$ can
be written as $\cS + \cA$ for some subset $\cS$ of quadruples from configuration $(1,1,1,1)$ which implies
by Lemma~\ref{lem:parallel1111} that $\cL'_5$ can be partitioned into pairwise disjoint parallel classes.
\end{proof}
The structure of the quadruples in $\cL'_1$, $\cL'_2$, $\cL'_3$, and $\cL'_4$ implies the following result.
\begin{corollary}
\label{cor:AinT1}
The quadruples of the subset $\cA$ are contained in $\cL'_5$ and form the quadruples from
configuration $(1,1,1,1)$ which are used in Type 1 when $t \equiv 1~(\text{mod}~4)$.
\end{corollary}

\begin{lemma}
\label{lem:like_OA9}
$~$
\begin{enumerate}
\item Each pair $\{ (y,0),(z,1) \}$, $y,z \in \Z_t$, is contained exactly $\frac{t-1}{2}$ times in the quadruples of $\cL'_2$.
Each pair $\{ (y,2),(z,3) \}$, $y,z \in \Z_t$, is contained exactly $\frac{t-1}{2}$ times in the quadruples of $\cL'_2$.

\item Each pair $\{ (y,0),(z,3) \}$, $y,z \in \Z_t$, is contained exactly $\frac{t-1}{2}$ times in the quadruples of $\cL'_3$.
Each pair $\{ (y,1),(z,2) \}$, $y,z \in \Z_t$, is contained exactly $\frac{t-1}{2}$ times in the quadruples of $\cL'_3$.

\item Each pair $\{ (y,0),(z,2) \}$, $y,z \in \Z_t$, is contained exactly $\frac{t-1}{2}$ times in the quadruples of $\cL'_4$.
Each pair $\{ (y,1),(z,3) \}$, $y,z \in \Z_t$, is contained exactly $\frac{t-1}{2}$ times in the quadruples of $\cL'_4$.
\end{enumerate}
\end{lemma}
\begin{proof}
The proof for $\cL'_2$ and $\cL'_3$ is exactly as in the proof of Lemma~\ref{lem:like_OA}.
Hence, we concentrate in the proof of the claim for $\cL'_4$.

First note that in $\cA$ there are $t$ quadruples and for any given $x \in \Z_t$ and $i \in \Z_4$, the element $(x,i)$
is contained in exactly one of the quadruples of $\cA$. Next, we consider the subset
$$
\{ \{ (i,0),(0,1),(i,2),(0,3)\} ~:~  1 \leq i \leq \frac{t-1}{2} \} + \cA .
$$
Each pair $\{(i,0), (i,2)\}$ and each pair $\{(j,1), (j,3)\}$, where $i,j \in \Z_t$, is contained in exactly
$\frac{t-1}{2}$ quadruples of this subset. Consider now the subset
$$
\{ \{ (j,0),(j,1),(0,2),(0,3)\} ~:~ j \in \Z_t \} + \{ \{ (i,0),(0,1),(i,2),(0,3)\} ~:~  1 \leq i \leq \frac{t-1}{2} \} + \cA .
$$
In this subset each pair $\{(i,0), (j,2)\}$, $i,j \in \Z_t$, and each pair $\{(k,1), (m,3)\}$, $k,m \in \Z_t$, is contained in exactly
$\frac{t-1}{2}$ quadruples. The same is true for the subset
$$
\{ (7,0),(0,1),(0,2),(0,3)\} +\{ \{ (j,0),(j,1),(0,2),(0,3)\} ~:~ j \in \Z_t \}
$$
$$
+ \{ \{ (i,0),(0,1),(i,2),(0,3)\} ~:~  1 \leq i \leq \frac{t-1}{2} \} + \cA .
$$

Consider now the two subsets
{\small
$$
\{ (7,0),(0,1),(0,2),(0,3)\} + \{ (0,0),(0,1),(0,2),(0,3)\}
+ \{ \{ (i,0),(0,1),(i,2),(0,3)\} ~:~  1 \leq i \leq \frac{t-1}{2} \} + \cA
$$
}
and
{\small
$$
\{ (7,0),(0,1),(0,2),(0,3)\} + \{ (-7,0),(-7,1),(0,2),(0,3)\}
+ \{ \{ (i,0),(0,1),(i,2),(0,3)\} ~:~  1 \leq i \leq \frac{t-1}{2} \} + \cA.
$$
}

These two subsets contain the same pairs $\{(i,0), (j,2)\}$, $\{(k,1), (m,3)\}$, $i,j,k,m \in \Z_t$, and their
multiplicity as in the two subsets
$$
\{ \{ (i,0),(0,1),(i,2),(0,3)\} ~:~  1 \leq i \leq \frac{t-1}{2} \} + \cA
$$
and
$$
\{ (7,0),(-7,1),(0,2),(0,3) \} + \{ \{ (i,0),(0,1),(i,2),(0,3)\} ~:~  1 \leq i \leq \frac{t-1}{2} \} + \cA.
$$

Thus, each pair $\{ (y,0),(z,2) \}$, $y,z \in \Z_t$, is contained
in exactly $\frac{t-1}{2}$ times in the quadruples of $\cL'_4$.
Similarly, each pair $\{ (y,1),(z,3) \}$, $y,z \in \Z_t$, is contained
in exactly $\frac{t-1}{2}$ times in the quadruples of $\cL'_4$.
\end{proof}

\begin{corollary}
\label{cor:Gsat9}
The subset $\cL'_2$ has the properties required for $\cG_1$ in Type 3 for odd $t$,
the subset $\cL'_3$ has the properties required for $\cG_2$ in Type 3 for odd $t$,
and the subset $\cL'_4$ has the properties required for $\cG_3$ in Type 3 for odd $t$.
These subsets form $3 \binom{t}{2}$ parallel classes of configuration $(1,1,1,1)$.
\end{corollary}

\begin{lemma}
\label{lem:GsatG5}
$~$
\begin{enumerate}
\item In $\cL'_2$ there exists a subset $S$ with exactly $t$ quadruples and the following properties.
The quadruples of $S$ have the form $\{ (i,0),(i,1),(j,2),(j,3)\}$.
For each $i \in \Z_t$, $S$ has exactly one quadruple of the form $\{ (i,0),(i,1),(j,2),(j,3)\}$.
For each $j \in \Z_t$, $S$ has exactly one quadruple of the form $\{ (i,0),(i,1),(j,2),(j,3)\}$.

\item In $\cL'_3$ there exists a subset $S$ with exactly $t$ quadruples and the following properties.
The quadruples of $S$ have the form $\{ (i,0),(j,1),(j,2),(i,3)\}$.
For each $i \in \Z_t$, $S$ has exactly one quadruple of the form $\{ (i,0),(j,1),(j,2),(i,3)\}$.
For each $j \in \Z_t$, $S$ has exactly one quadruple of the form $\{ (i,0),(j,1),(j,2),(i,3)\}$.

\item In $\cL'_4$ there exists a subset $S$ with exactly $t$ quadruples and the following properties.
The quadruples of $S$ have the form $\{ (i,0),(j,1),(i,2),(j,3)\}$.
For each $i \in \Z_t$, $S$ has exactly one quadruple of the form $\{ (i,0),(j,1),(i,2),(j,3)\}$.
For each $j \in \Z_t$, $S$ has exactly one quadruple of the form $\{ (i,0),(j,1),(i,2),(j,3)\}$.
\end{enumerate}
\end{lemma}
\begin{proof}
We define the subset $S$ for each one of the subsets $\cL'_2$, $\cL'_3$, and $\cL'_4$.
For $\cL'_2$ let
$$
S \triangleq  \{ (1,0),(1,1),(0,2),(0,3)\} +\cA .
$$
For $\cL'_3$ let
$$
S \triangleq \{ (-1,0),(0,1),(0,2),(0,3)\} +  \{ (\frac{t+1}{2},0),(\frac{t+1}{2},1),(0,2),(0,3)\}
$$
$$
+ \{ (\frac{t+1}{2},0),(0,1),(\frac{t+1}{2},2),(0,3)\} +\cA = \{ (0,0),(\frac{t+1}{2},1),(\frac{t+1}{2},2),(0,3)\} +\cA.
$$
For $\cL'_4$ let
$$
S \triangleq  \{ (1,0),(0,1),(1,2),(0,3)\} +\cA .
$$
It is readily verified that each subset $S$ is contained in the related set $\cL'_2$, $\cL'_3$, and $\cL'_4$,
and also $S$ satisfy the requirements of the claim.
\end{proof}
\begin{corollary}
\label{cor:rmT3T5}
The subsets $\cL'_2$, $\cL'_3$, and $\cL'_4$, without the subset $S$ (of each one separately),
satisfy the properties required for Type 3 when $t \equiv 1~(\text{mod}~4)$.
These subsets of quadruples from Group 5 which are used in Type 3 form $3 \binom{t}{2} -3$ parallel classes
of configuration $(1,1,1,1)$ defined in Type~5.
\end{corollary}

We continue with this case, where the parallel classes are taken as defined in Section~\ref{sec:const_par},
with several exceptions. These exceptions occur since quadruples from Groups 3 and 5 are
used in Type 1. Furthermore, quadruples of Group 5 are also used in Type 3 and in Type 4.
Hence, our enumeration should be made carefully.
By Lemma~\ref{lem:T4t1} there are $\binom{t+2}{3}$ parallel classes of Type 1.
By Lemma~\ref{lem:rm22a31a1111f4} these parallel classes contain $3(t-1)t$ quadruples
from Group 3 and $t$ quadruples from Group 5.
By Lemma~\ref{lem:T31} we have $(4t-3)\binom{t-1}{2}$ parallel classes of Type 2.
Lemma~\ref{lem:rm22a31a1111f4}, Lemma~\ref{lem:T4t1}, and Lemma~\ref{lem:T31}, imply
that each quadruple from Group 1 and Group 2 is contained in exactly one of the
parallel classes of Type 1 or Type~2.
By Lemma~\ref{lem:T22odd} there are $3 t \binom{t}{2}$ parallel classes of Type 3.
By Lemma~\ref{lem:T22_t1}, the $3(t-1)t$ quadruples of $3t$ parallel classes were already used in Type 1.
By Lemma~\ref{lem:T22odd} each one of these $3 t \binom{t}{2} -3t$ parallel classes contains one quadruple from Group 5.
By Lemma~\ref{lem:T211odd} there are $6 t^3$ parallel classes of Type 4 and
each one contains one quadruple from Group 5.
There are $t^3$ parallel classes of Type 5, from which by Corollary~\ref{cor:AinT1} implies that
one parallel class is used in Type 1.
Combining this with Lemma~\ref{lem:allG5B}, Corollary~\ref{cor:Hsat}, and Corollary~\ref{cor:rmT3T5},
imply that there are $t^3 -1 - 3 \binom{t}{2} +3 - 6 t^2$ parallel classes of Type 5.
Taking all these parallel classes together, there are a total of
$$
\binom{t+2}{3} + (4t-3)\binom{t-1}{2} + 3 t \binom{t}{2} -3t + 6  t^3 + t^3 -1 - 3 \binom{t}{2} +3 - 6 t^2  = \binom{4t-1}{3}
$$
parallel classes as required. Thus, if there exist a $\mathrm{BP}(t+3,4)$ and a $\mathrm{BP}(t,3)$, then there exists a~$\mathrm{BP}(4t,4)$.


The last part of the proof is to give some required initial conditions for our recursive quadrupling
constructions. The only part of our proof that required some initial values for $t$ is for
$t \equiv 3$ which works only from $t \geq 15$.
Hence, an initial condition is required are for $\mathrm{BP}(12,4)$.
There are several ways to construct $\mathrm{BP}(12,4)$.
These constructions are left as an exercise for the interested reader.

Thus, we have completed the proof of Theorem~\ref{thm:main}.

\section{Efficiency of the Constructions}
\label{sec:efficient}

Recall the general definition of the Baranyai partition problem.
Let $X$ be a finite set of size $n$ and $k$ be an integer, where $0\leq k\leq |X|$.
We are interested in a $\mathrm{BP}(n,k)$, $\cP_1 \sqcup \cP_2 \sqcup \cdots \sqcup \cP_{\binom{n-1}{k-1}}$,
where $\cP_i$ is a parallel class and $\sqcup$ is used to denote disjoint union. Such a partition is represented
by an $(n/k)\times \binom{n-1}{k-1}$ matrix $\sf M$, where the $i$-th column of $\sf M$ contains the
$n/k$ blocks in $\cP_i$. We call $\sf M$ an $(n,k)$-Baranyai matrix.

There are several algorithmic problems which can be considered for Baranyai partitions. \\

\noindent {\sc Listing Problem:} Given $n$ and $k$, output an $(n,k)$-Baranyai matrix.  \\

\noindent {\sc Column Generation Problem:} Given $n$, $k$, and $1\leq i\leq \binom{n-1}{k-1}$, output
the $i$-th column, COL$(n,k,i)$, of an $(n,k)$-Baranyai matrix. \\

\noindent {\sc Entry Problem:} Given $n$, $k$, $1\leq i\leq \binom{n-1}{k-1}$, and $1\leq j\leq n/k$,
output the $(i,j)$-th entry, ENT$(n,k,i,j)$, of an $(n,k)$-Baranyai matrix. \\

It is clear that the computational complexity is highest for the {\sc Listing Problem}, followed by
the {\sc Column Generation Problem}, and then the {\sc Entry Problem}. But we also have
\begin{equation*}
\text{
{\sc Listing Problem} $\leq^P$ {\sc Column Generation Problem} $\leq^P$ {\sc Entry Problem},
}
\end{equation*}
where $\leq^P$ denotes polynomial-time Turing reduction.
For general fixed $k$, linear-time algorithms (in the number of $k$-subsets) for the
{\sc Listing Problem} is not known. The known proofs of Baranyai's theorem requires solving network flow problems
and the implementation for Baranyai partition is of a high degree polynomial ($O(n^{3k-2})$ for
the implementation presented in~\cite{BrSc79}; see also~\cite[pp. 536 -- 542]{LiWi01}).
The {\sc Column Generation Problem} and the {\sc Entry Problem} were not considered in previous works
and it is unclear how the network flow algorithm can provide more efficient solutions these problems.
Our methods are more efficient than the solutions via the network flow techniques for these problems. Our solution to the
{\sc Listing Problem} problem is linear in the number of 4-subsets, i.e. its complexity is $O(n^4)$
compared to $O(n^{10})$ using network flow. Our solution to the {\sc Column Generation Problem}
has complexity $O(n)$ and the same for the {\sc Entry Problem}. The efficiency of our constructions
is based on the enumerative coding of Cover~\cite{Cov73}.

We demonstrate the enumerative coding on a $\mathrm{BP}(2^n,4)$ for any $n \geq 2$.
This Baranyai's partition has $\binom{2^n-1}{3}$ parallel classes, each one with $2^{n-2}$ quadruples.
The construction of this Baranyai's partition is presented and analysed in Section~\ref{sec:doubling}.
The construction has three types of parallel classes, Type $\cS$, Type $\cT$, and Type $\cF$.
In Type $\cS$ there are $8 \binom{2^{n-1}-1}{3}$ parallel classess, in Type $\cT$ there are $4(2^{n-1}-1)(2^{n-1}-2)$ parallel classes,
and in Type $\cF$ there are $2^{n-1}-1$ parallel classes. These numbers of parallel classes will be taken into account in
the enumerative coding.

We distinguish now between the three algorithmic problems.

\noindent
The {\sc Column Generation Problem}:

Assume we look for COL$(2^n,4,\beta)$. We distinguish between three cases.

\begin{enumerate}
\item $1 \leq \beta \leq 8 \binom{2^{n-1} -1}{3}$.

In this case the parallel class is of Type $\cS$ (see Section~\ref{sec:cS}).
For this type the recursion generates eight parallel classes for each parallel class
of $\mathrm{BP}(2^{n-1},4)$. Hence, we compute COL$(2^{n-1},4,\lceil \frac{\beta}{8} \rceil)$
and assume that with this recursion we obtain the quadruples $X_1,X_2,\ldots,X_{2^{n-3}}$.
If $\beta = i + 8(\lceil \frac{\beta}{8} \rceil -1)+1$, then COL$(2^n,4,\beta)$ is obtained by applying
the $i$-th parallel class ($i \in \Z_8$ - see Section~\ref{sec:cS}), in the construction for Type $\cS$,
on the quadruples $X_1,X_2,\ldots,X_{2^{n-3}}$ as presented in Section~\ref{sec:cS}.

\item $8 \binom{2^{n-1} -1}{3} < \beta \leq 8 \binom{2^{n-1} -1}{3} + 4(2^{n-1}-1)(2^{n-1}-2)$.

In this case the parallel class is of Type $\cT$ (see Section~\ref{sec:cT}).
For this an efficient construction for resolvable $\mathrm{SQS}(2^{n-1})$ is required. Such a construction
is not difficult to find and in general it can be obtained by considering the constructions in~\cite{Har87,LiZh05}.
In general the complexity for the construction of the parallel class of the resolution for $\mathrm{SQS}(t)$ is $O(t)$.
The number of parallel classes in such a resolution of $\mathrm{SQS}(2^{n-1})$ is $\frac{(2^{n-1}-1)(2^{n-2}-1)}{3}$.
For this type the recursion generates twenty four parallel classes for each parallel class
of $\mathrm{SQS}(2^{n-1})$.
The parallel class which is required is number $\lceil \frac{\beta - 8 \binom{2^{n-1} -1}{3}}{24} \rceil$.
Assume that this parallel class contains the quadruples $X_1,X_2,\ldots,X_{2^{n-3}}$.
If $\beta = i + 24 ( \lceil \frac{\beta - 8 \binom{2^{n-1} -1}{3}}{24} \rceil -1)  +1$,
then COL$(2^n,4,\beta)$ is obtained by applying
the $i$-th parallel class ($i \in \Z_{24}$, where a translation from
$\cR_{\alpha,\beta,\gamma}$ should be made - see Section~\ref{sec:cT}), in the construction for Type $\cT$,
on the quadruples $X_1,X_2,\ldots,X_{2^{n-3}}$ as presented in Section~\ref{sec:cT}.

\item $8 \binom{2^{n-1} -1}{3} + 4(2^{n-1}-1)(2^{n-1}-2) < \beta$.

In this case the the parallel class is of Type $\cF$ (see Section~\ref{sec:cF}).
If $i = \beta - 8 \binom{2^{n-1} -1}{3} - 4 (2^{n-1}-1)(2^{n-1}-2) -1$, then
COL$(2^n,4,\beta)$ is the following set of blocks.
$$
\{ \{ (a,0),(b,0),(a,1),(b,1) \} ~:~  \{ a,b \} \in \cF_i \},
$$
where $\{ \cF_0,\cF_1,\ldots,\cF_{2^{n-1}-2} \}$ is any one-factorization of $K_{2^{n-1}}$.

\end{enumerate}

The complexity of this enumerative coding of the {\sc Column Generation Problem} for $\mathrm{BP}(t,4)$ is $O(t)$.
This immediately implies that an algorithm for the {\sc Entry Problem} can be given with the same
complexity or even a better one. The algorithm can be simpler for the {\sc Listing Problem}, but by applying
the {\sc Column Generation Problem} for all the columns of Baranyai matrix to obtain a solution for the {\sc Listing Problem}
implies complexity $O(n^4)$ which cannot be improved since this is the size of the list (matrix).

\section{Conclusions and Future Work}
\label{sec:conclude}

We have presented a few explicit and efficient recursive quadrupling constructions for $\mathrm{BP}(4t,4)$, where
$t \equiv 0,3,4,6,8,9~(\text{mod}~12)$. In a follow-up paper~\cite{CEKVW20}
we provide a quadrupling construction for the remaining cases, i.e., when $t \equiv 1,2,5,7,10,11~(\text{mod}~12)$.
Once a construction for $t \equiv 1,5,7,11~(\text{mod}~12)$ will be presented,
Theorem~\ref{thm:4_8m12} can be used to apply a quadrupling construction for the
last two cases when $t \equiv 2$ or $10~(\text{mod}~12)$.
A future intriguing work is to present an efficient construction for $\mathrm{BP}(n,5)$
for a sequence of infinite values of $n$.

\end{document}